\documentclass[11pt,reqno]{amsart}

\usepackage[utf8]{inputenc}
\usepackage{amsfonts}
\usepackage{amsmath}
\usepackage{amssymb}
\usepackage{amsthm}
\usepackage{mathrsfs}
\usepackage{stmaryrd}
\usepackage{color}
\usepackage[english]{babel}
\usepackage{fontenc}
\usepackage{url}
\usepackage{graphicx}
\usepackage{hyperref}
\usepackage{caption}
\usepackage{epstopdf}
\usepackage{hyphenat}
\usepackage{float}
\usepackage{indentfirst}
\usepackage[export]{adjustbox}
\usepackage{tikz}
\usetikzlibrary{matrix,arrows,decorations.pathmorphing}

\allowdisplaybreaks

\newtheorem{theorem}{Theorem}[section]
\newtheorem{lemma}[theorem]{Lemma}
\newtheorem{proposition}[theorem]{Proposition}
\newtheorem{corollary}[theorem]{Corollary}
\newtheorem*{theorem*}{Theorem}

\theoremstyle{definition}
\newtheorem{observation}[theorem]{Observation}
\newtheorem{remark}[theorem]{Remark}
\newtheorem{example}[theorem]{Example}
\newtheorem{definition}[theorem]{Definition}
\newtheorem*{acknowledgments}{Acknowledgments}

\DeclareMathOperator{\Bl}{Bl}
\DeclareMathOperator{\Pic}{Pic}
\DeclareMathOperator{\Km}{Km}
\DeclareMathOperator{\NS}{NS}
\DeclareMathOperator{\SL}{SL}
\DeclareMathOperator{\Num}{Num}
\DeclareMathOperator{\Aut}{Aut}
\DeclareMathOperator{\rk}{rk}
\DeclareMathOperator{\topo}{top}
\DeclareMathOperator{\diag}{diag}
\DeclareMathOperator{\Hom}{Hom}
\DeclareMathOperator{\Supp}{Supp}
\DeclareMathOperator{\id}{id}
\makeatletter
\def\l@subsection{\@tocline{2}{0pt}{2.5pc}{5pc}{}}
\makeatother

\author{Luca Schaffler}
\title{K3 surfaces with $\mathbb{Z}_2^2$ symplectic action}
\date{}

\begin{document}

\maketitle

\begin{abstract}
Let $G$ be a finite abelian group which acts symplectically on a K3 surface. The N\'eron-Severi lattice of the projective K3 surfaces admitting $G$ symplectic action and with minimal Picard number is computed by Garbagnati and Sarti. We consider a $4$-dimensional family of projective K3 surfaces with $\mathbb{Z}_2^2$ symplectic action which do not fall in the above cases. If $X$ is one of these K3 surfaces, then it arises as the minimal resolution of a specific $\mathbb{Z}_2^3$-cover of $\mathbb{P}^2$ branched along six general lines. We show that the N\'eron-Severi lattice of $X$ with minimal Picard number is generated by $24$ smooth rational curves, and that $X$ specializes to the Kummer surface $\Km(E_i\times E_i)$. We relate $X$ to the K3 surfaces given by the minimal resolution of the $\mathbb{Z}_2$-cover of $\mathbb{P}^2$ branched along six general lines, and the corresponding Hirzebruch-Kummer covering of exponent $2$ of $\mathbb{P}^2$.
\end{abstract}
\maketitle


\section{Introduction}
Let $X$ be a K3 surface over $\mathbb{C}$. A subgroup of the automorphism group of $X$ acts symplectically on $X$ if the induced action on $H^0(X,\omega_X)$ is the identity. Finite abelian groups of automorphisms acting symplectically on K3 surfaces were classified by Nikulin in \cite{nikulin79}, and by Mukai in the non-commutative case (see \cite{mukai} and also \cite{xiao}).

Let $G$ be a finite abelian group that acts symplectically on a K3 surface $X$. If $\Omega_G$ denotes the orthogonal complement in $H^2(X;\mathbb{Z})$ of the fixed sublattice $H^2(X;\mathbb{Z})^G$, then $\Omega_G$ is a negative definite primitive sublattice of the N\'eron-Severi lattice $\NS(X)$. All the possible lattices $\Omega_G$ and their orthogonal complements in $H^2(X;\mathbb{Z})$ are computed in \cite{garbagnatisarti07,garbagnatisarti09}. If $X$ is projective, then the Picard number $\rho(X)$ satisfies $\rho(X)\geq\rk(\Omega_G)+1$.

The projective K3 surfaces $X$ admitting a $G$ symplectic action form a family of dimension $19-\rk(\Omega_G)$ (see \cite[Remark 6.2]{garbagnatisarti09}). If $\rho(X)=\rk(\Omega_G)+1$, then the lattices $\NS(X)$ are computed in \cite[Proposition 6.2]{garbagnatisarti09}. The case $G=\mathbb{Z}_2^4$ received special attention: in \cite{garbagnatisarti16}, among other results, the authors compute $\NS(\widetilde{X/G})$, where $\widetilde{X/G}$ is the minimal resolution of $X/G$, if $\rho(\widetilde{X/G})=16$.

There is a $7$-dimensional family of projective K3 surfaces with $\mathbb{Z}_2^2$ symplectic action, and in this paper we analyze the $4$-dimensional subfamily which arises as follows. Consider six general lines in $\mathbb{P}^2$ and divide them into three pairs. Consider the chain of double covers
\begin{equation*}
X_3\xrightarrow{\mathbb{Z}_2}X_2\xrightarrow{\mathbb{Z}_2}X_1\xrightarrow{\mathbb{Z}_2}\mathbb{P}^2,
\end{equation*}
where the first cover is branched along the first pair of lines, the second cover is branched along the preimage of the second pair of lines, and so on. The minimal resolution of $X_3$ is a projective K3 surface $X$, which we call a \emph{triple-double} K3 surface (see Section~\ref{definitiontriple-doublek3surfaces}). Equivalently, $X$ can be viewed as an appropriate $\mathbb{Z}_2^3$-cover of $\Bl_3\mathbb{P}^2$, which denotes the blow up of $\mathbb{P}^2$ at three general points. The surface $X$ admits a $\mathbb{Z}_2^2$ symplectic action and an Enriques involution (see Proposition~\ref{involutiontriple-doublek3surface}). The minimal Picard number for $X$ is $16$, as we show in Section~\ref{theneronseverilatticeofatriple-doublek3surface} (this happens for a very general $X$). Moduli compactifications of the corresponding $4$-dimensional family of Enriques surfaces are studied in \cite{oudompheng,schaffler}. The main goal of this paper is to compute the lattice $\NS(X)$ for $X$ with minimal Picard number, and relate it to the geometry of $X$.
\begin{theorem}
\label{mainresultinintro}
Let $X$ be a triple-double K3 surface with minimal Picard number. Then the following hold:
\begin{itemize}
\item[(i)] The N\'eron-Severi lattice $\NS(X)$ has rank $16$ and is generated by the irreducible components of the preimage of the $(-1)$-curves on $\emph{Bl}_3\mathbb{P}^2$. The dual graph of this configuration of $24$ smooth rational curves is shown in Figure~\ref{dualgraph24smoothrationalcurves} (see Theorem~\ref{mainresultNSgeneratedby24smoothrationalcurves});
\item[(ii)] $\NS(X)$ has discriminant group $\mathbb{Z}_2^2\oplus\mathbb{Z}_4^2$ and is isometric to $U\oplus E_8\oplus Q$, where $Q$ is the lattice in Lemma~\ref{thesublatticesSandQ}(iii). An explicit $\mathbb{Z}$-basis of $\NS(X)$ which realizes it as a direct sum of $U,E_8$, and $Q$ can be found in Remark~\ref{exampleofbasiswhichgivessplitting};
\item[(iii)] The transcendental lattice $T_X$ is isometric to $U\oplus U(2)\oplus\langle-4\rangle^{\oplus2}$ (see Proposition~\ref{transcendentallatticetriple-doublek3surface});
\item[(iv)] The Kummer surface $\Km(E_i\times E_i)$ (see \cite{keumkondo}) appears as a specialization of the $4$-dimensional family of triple-double K3 surfaces (see Theorem~\ref{explicitlinearrangementthatgiveskmeixei}). The line arrangement in $\mathbb{P}^2$ that gives rise to this Kummer surface is shown in Figure~\ref{linearrangementKm(EixEi)};
\item[(v)] Let $\iota$ be an involution on $X$ coming from the $\mathbb{Z}_2^2$ symplectic action, and denote by $X'$ the minimal resolution of $X/\iota$. For $X'$ with minimal Picard number, the N\'eron-Severi lattice of $X'$ has rank $16$ and discriminant group $\mathbb{Z}_2^4\oplus\mathbb{Z}_4^2$. An explicit $\mathbb{Z}$-basis for $\NS(X')$ is given in Theorem~\ref{allaboutnsx'andtx'}. Finally, the transcendental lattice $T_{X'}$ is isometric to $U(2)^{\oplus2}\oplus\langle-4\rangle^{\oplus2}$.
\end{itemize}
\end{theorem}

The N\'eron-Severi lattice of a triple-double K3 surface $X$ with $\rho(X)=16$ is not included in \cite{garbagnatisarti07,garbagnatisarti09,garbagnatisarti16} for the following reasons:
\begin{itemize}
\item If $G$ is a finite abelian group acting symplectically on a projective K3 surface, then $16=\rk(\Omega_G)+1$ if and only if $G\cong\mathbb{Z}_2^4$ (see \cite[Proposition 5.1]{garbagnatisarti09});
\item $X$ does not admit $\mathbb{Z}_2^4$ symplectic action (see Proposition~\ref{XnoZ24andnoZ23symplecticactions}(i));
\item $X$ is not isomorphic to the minimal resolution of the quotient of a K3 surface by a symplectic action of the group $\mathbb{Z}_2^4$ (see Proposition~\ref{XnoZ24andnoZ23symplecticactions}(ii)).
\end{itemize}

Triple-double K3 surfaces are closely related to two classical examples of K3 surfaces. The first one is the Hirzebruch-Kummer covering $Y$ of exponent $2$ of $\mathbb{P}^2$ branched along six general lines (see \cite[Section 4.3]{catanese}). The covering map $Y\rightarrow\mathbb{P}^2$ is a $\mathbb{Z}_2^5$-cover, and there is a subgroup of $\mathbb{Z}_2^5$ isomorphic to $\mathbb{Z}_2^4$ which acts symplectically on $Y$. The second example is the minimal resolution of the double cover of $\mathbb{P}^2$ branched along six general lines, which we denote by $Z$. Now let $X$ be a triple-double K3 surface. Given the corresponding six lines in $\mathbb{P}^2$, we can consider $Y$ and $Z$ as above. Then in Proposition~\ref{triple-doublek3surfacesandthehirzebruchkummercover} we show that $X$ is isomorphic to $\widetilde{Y/\mathbb{Z}_2^2}$ for an appropriate subgroup $\mathbb{Z}_2^2<\mathbb{Z}_2^4$, and the minimal resolution of the quotient of $X$ by the leftover $\mathbb{Z}_2^4/\mathbb{Z}_2^2\cong\mathbb{Z}_2^2$ symplectic action gives rise to $Z$. Summarizing, we have the following diagram:
\begin{center}
\begin{tikzpicture}[>=angle 90]
\matrix(a)[matrix of math nodes,
row sep=2em, column sep=2em,
text height=1.5ex, text depth=0.25ex]
{Y&Y/\mathbb{Z}_2^2&X&X/\mathbb{Z}_2^2&Z\\
&&\mathbb{P}^2.&&\\};
\path[->] (a-1-1) edge node[left]{}(a-1-2);
\path[->] (a-1-3) edge node[above left]{}(a-1-2);
\path[->] (a-1-3) edge node[left]{}(a-1-4);
\path[->] (a-1-5) edge node[above left]{}(a-1-4);
\path[->] (a-1-1) edge node[above left]{}(a-2-3);
\path[->] (a-1-3) edge node[above left]{}(a-2-3);
\path[->] (a-1-5) edge node[above left]{}(a-2-3);
\end{tikzpicture}
\end{center}
The N\'eron-Severi lattices $\NS(Y),\NS(Z)$ for $\rho(Y)=\rho(Z)=16$ are well known, and the interaction between them is described in \cite{garbagnatisarti16} (see \cite[Section 5]{kloosterman} for $\NS(Z)$).

We remark that similar problems are considered in \cite{bouyer}, but different families of K3 surfaces are studied there. We also mention \cite{nikulin17}, where Nikulin classified the \emph{main part} of the N\'eron-Severi lattice of an arbitrary K\"ahlerian K3 surface (see \cite[Section 3]{nikulin17}) with large enough group of symplectic automorphisms (see \cite[Theorem 2]{nikulin17} for a precise statement). There are more examples of K3 surfaces that are the minimal resolution of a $\mathbb{Z}_2^n$-cover of $\mathbb{P}^2$ branched along six general lines and that are different from $X,X',Y,Z$ (see the notation introduced above). These other K3 surfaces will be studied in a different paper.

In Section~\ref{preliminariesonk3surfacesandsymplecticautomorphisms} we recall some preliminary facts about K3 surfaces, symplectic automorphisms, and even lattices. In Section~\ref{triple-doublek3surfaces} we define triple-double K3 surfaces characterizing them in several different ways. We classify the involutions coming from the $\mathbb{Z}_2^3$-action of the cover. In the same section we also analyze the configuration of $24$ smooth rational curves on a triple-double K3 surface $X$ coming from the preimage of the $(-1)$-curves under the covering map $X\rightarrow\Bl_3\mathbb{P}^2$. In Section~\ref{theneronseverilatticeofatriple-doublek3surface} we show that the sublattice of $\NS(X)$ generated by these curves equals $\NS(X)$ for $X$ with minimal Picard number. We also compute $T_X$ and show that $\Km(E_i\times E_i)$ is a specialization of this $4$-dimensional family. In Section~\ref{three4dimensionalfamiliesofk3surfacesrelated} we relate triple-double K3 surfaces with the Hirzebruch-Kummer covering of exponent $2$ of $\mathbb{P}^2$ branched along six general lines. Finally, in Section~\ref{NSX'andTX'computed} we compute $\NS(X')$ and $T_{X'}$, where $X'$ is the minimal resolution of the quotient of a triple-double K3 surface by an involution coming from the $\mathbb{Z}_2^2$ symplectic action and $\rho(X')=16$. We work over $\mathbb{C}$.

\begin{acknowledgments}
I would like to thank my advisor, Valery Alexeev, for his suggestions and for motivating me in studying this family of K3 surfaces. I am also grateful to Eyal Markman and Robert Varley for many interesting conversations related to this project. A special thanks to Alice Garbagnati and Alessandra Sarti for their helpful feedback. In particular, thanks to Alice Garbagnati for pointing out the copy of $U\oplus E_8$ in Figure~\ref{subgraphextendede8plussection}, which simplified many proofs, and thanks to Alessandra Sarti for suggesting the connection with $\Km(E_i\times E_i)$. Many thanks to Simon Brandhorst and Klaus Hulek for helpful conversations at the Leibniz Universit\"at Hannover. I would like to thank the anonymous referee for very useful and constructive comments. I gratefully acknowledge financial support from the Dissertation Completion Award at the University of Georgia, the NSF grant DMS-1603604, and the Research and Training Group in Algebra, Algebraic Geometry, and Number Theory, at the University of Georgia.
\end{acknowledgments}


\tableofcontents


\section{Preliminaries: lattice theory and symplectic automorphisms of K3 surfaces}
\label{preliminariesonk3surfacesandsymplecticautomorphisms}


\subsection{Even lattices and the discriminant quadratic form}
A \emph{lattice} $L=(L,b_L)$ is a finitely generated free abelian group $L$ together with a symmetric bilinear form $b_L\colon L\times L\rightarrow\mathbb{Z}$. In what follows, we consider \emph{non-degenerate} lattices, which means $b_L$ is a non-degenerate symmetric bilinear form, and we always assume so. If $e_1,\ldots,e_n$ is a $\mathbb{Z}$-basis for $L$, then the \emph{Gram} matrix of $L$ associated to the chosen $\mathbb{Z}$-basis is the matrix $(b_L(e_i,e_j))_{1\leq i,j\leq n}$. The determinant of a Gram matrix of $L$ does not depend on the choice of $\mathbb{Z}$-basis and is called the \emph{discriminant} of $L$. We say that $L$ is \emph{even} if $b_L(x,x)\in2\mathbb{Z}$ for all $x\in L$, and \emph{odd} otherwise. Lattices of discriminant $\pm1$ are called \emph{unimodular}. For the classification of unimodular lattices we refer to \cite[Chapter V]{serre}. Denote by $U$ (resp. $E_8$) the unique even unimodular lattice of signature $(1,1)$ (resp. $(0,8)$).

Let $L$ be a lattice. The dual $L^*=\Hom_\mathbb{Z}(L,\mathbb{Z})$ comes with an induced $\mathbb{Q}$-valued symmetric bilinear form $b_{L^*}$. $L$ embeds into $L^*=\Hom_\mathbb{Z}(L,\mathbb{Z})$ because $b_L$ is non-degenerate, and the quotient $A_L=L^*/L$ is called the \emph{discriminant group} of $L$. We denote by $\ell(A_L)$ the minimum number of generators of $A_L$. The lattice $L$ is \emph{$2$-elementary} if $A_L$ is isomorphic to $\mathbb{Z}_2^\alpha$ for some $\alpha$. The discriminant group $A_L$ comes equipped with the quadratic form
\begin{align*}
q_L\colon A_L&\rightarrow\mathbb{Q}/2\mathbb{Z}\\
x+L&\mapsto b_{L^*}(x,x)~\textrm{mod}~2\mathbb{Z},
\end{align*}
which is called the \emph{discriminant} quadratic form of $L$.

A sublattice $S\subseteq L$ is called \emph{primitive} if the quotient $L/S$ is torsion free, and $L$ is called an \emph{overlattice} of $S$ if $L/S$ is finite. We recall the following standard results for the reader's convenience.
\begin{theorem}{\cite[Corollary 1.13.3]{nikulin80}}
\label{uniquenessofevenlattices}
Let $L$ be an even indefinite lattice of signature $(t_+,t_-)$ such that $t_++t_-\geq2+\ell(A_L)$. Then $L$ is unique up to isometry.
\end{theorem}
\begin{theorem}{\cite[Corollary 1.13.5]{nikulin80}}
\label{splittingofevenlattices}
Let $L$ be an even lattice of signature $(t_+,t_-)$.
\begin{itemize}
\item[(i)] If $t_+\geq1,t_-\geq8$ and $t_++t_-\geq9+\ell(A_L)$, then $L\cong E_8\oplus P$ for some lattice $P$;
\item[(ii)] If $t_+\geq1,t_-\geq1$ and $t_++t_-\geq3+\ell(A_L)$, then $L\cong U\oplus P$ for some lattice $P$.
\end{itemize}
\end{theorem}
\begin{theorem}{\cite[Proposition 1.4.1(a)]{nikulin80}}
\label{overlatticesinbijectionwithisotropicsubgroups}
Let $L$ be an even lattice. Then there is a $1$-to-$1$ correspondence between overlattices of $L$ and subgroups of $A_L$ which are isotropic with respect to the discriminant quadratic form $q_L$.
\end{theorem}
\begin{theorem}{\cite[\S1, $5^\circ$]{nikulin80}}
\label{sublatticeandperpsamediscriminantgroup}
Let $L$ be an even unimodular lattice and $S\subseteq L$ a primitive sublattice. Then $S^\perp\subseteq L$ is also primitive, and the two discriminant groups $A_S,A_{S^\perp}$ are isomorphic. Moreover, $q_S=-q_{S^\perp}$.
\end{theorem}
\begin{theorem}{\cite[Theorem 3.6.2]{nikulin80}}
\label{uniqueness2-elementarylattices}
An indefinite $2$-elementary even lattice $L$ is determined, up to isometry, by its signature, $\ell(A_L)$, and $\delta(L)$ (for the definition of $\delta(L)$ we refer to \cite[\S3, $6^\circ$]{nikulin80}).
\end{theorem}


\subsection{K3 surfaces}
\label{k3surfacesbasics}
A \emph{K3 surface} $X$ is a smooth irreducible projective $2$-dimensional variety with $K_X\sim0$ and $h^1(X,\mathcal{O}_X)=0$. On a K3 surface $X$, numerical algebraic, and linear equivalence between divisors coincide (see \cite[Chapter 1, Proposition 2.4]{huybrechts}), and therefore
\begin{equation*}
\Pic(X)=\NS(X)=\Num(X).
\end{equation*}
In the case of a K3 surface $X$ one has that $\NS(X)$ is a primitive sublattice of $H^2(X;\mathbb{Z})$ (the fact that $h^1(X,\mathcal{O}_X)=0$ implies that $\NS(X)$ embeds into $H^2(X;\mathbb{Z})$, and $H^2(X;\mathbb{Z})/\NS(X)$ embeds into $H^2(X,\mathcal{O}_X)\cong\mathbb{C}$). Recall that $H^2(X;\mathbb{Z})$ is an even unimodular lattice of signature $(3,19)$, and is therefore isometric to $U^{\oplus3}\oplus E_8^{\oplus2}$. It follows that the \emph{transcendental} lattice $T_X=\NS(X)^\perp$ has the same discriminant group as $\NS(X)$ (see Theorem~\ref{sublatticeandperpsamediscriminantgroup}). The \emph{Picard number} of a K3 surface $X$ is the rank of $\NS(X)$, and it is denoted by $\rho(X)$.


\subsection{Symplectic automorphisms of K3 surfaces}
\label{definitionnikulininvolution}
Let $X$ be a K3 surface. An automorphism $f$ of $X$ is called \emph{symplectic} if the induced map $f^*\colon H^0(X,\omega_X)\rightarrow H^0(X,\omega_X)$ is the identity. The effective (left) action of a group $G$ on the K3 surface $X$ is called \emph{symplectic} if for all $g\in G$ the automorphism of $X$ given by $x\mapsto g\cdot x$ is symplectic. It is well known that a finite group $G$ acts symplectically on $X$ if and only if the minimal resolution $\widetilde{X/G}$ of $X/G$ is again a K3 surface (see \cite[(8.10) Proposition(1)]{mukai} or \cite[Theorem 0.4.2]{garbagnati}).

Let $G$ be a finite abelian group. In \cite{nikulin79} Nikulin shows that $G$ acts symplectically on a K3 surface $X$ if and only if $G$ is one of the following groups:
\begin{gather*}
\mathbb{Z}_n,~2\leq n\leq8,~\mathbb{Z}_m^2,~m=2,3,4,\\
\mathbb{Z}_2\oplus\mathbb{Z}_4,~\mathbb{Z}_2\oplus\mathbb{Z}_6,~\mathbb{Z}_2^\ell,~\ell=3,4.
\end{gather*}
Moreover, the action of $G$ on $H^2(X;\mathbb{Z})$ is unique up to isometry, hence it depends only on $G$ and not on the K3 surface $X$. Denote by $\Omega_G$ the orthogonal complement of the fixed sublattice $H^2(X;\mathbb{Z})^G$. The next theorem briefly summarizes many results of Garbagnati and Sarti which are fundamental for the current paper.
\begin{theorem}{\cite{garbagnatisarti09}}
\label{garbagnatisartiresults}
Let $G$ be a finite abelian group that acts symplectically on a K3 surface $X$. Then the orthogonal complement of $\Omega_G$ in $H^2(X;\mathbb{Z})$ is computed in \cite[Proposition 5.1]{garbagnatisarti09}. Moreover, if $\rho(X)=\rk(\Omega_G)+1$, then $\NS(X)$ is given in \cite[Proposition 6.2]{garbagnatisarti09}.
\end{theorem}
\begin{remark}
The results in Theorem~\ref{garbagnatisartiresults} for $G=\mathbb{Z}_2$ go back to \cite{vangeemensarti,morrison}, and the cases $G=\mathbb{Z}_p$ for $p=3,5,7$ appeared in \cite{garbagnatisarti07}.
\end{remark}
\begin{definition}
A symplectic automorphism of order $2$ on a K3 surface $X$ is called \emph{symplectic involution} or \emph{Nikulin involution}. A Nikulin involution on $X$ has exactly eight fixed points (see \cite[Section 0.1]{mukai}), and in Proposition~\ref{involutionwithexactlyeightisolatedfixedpointsissymplectic} (which is well known) we show that an involution on $X$ with exactly eight fixed points is necessarily symplectic.
\end{definition}
\begin{proposition}
\label{involutionwithexactlyeightisolatedfixedpointsissymplectic}
Let $X$ be a K3 surface and let $\iota$ be an involution on $X$ with exactly eight fixed points. Then $\iota$ is a Nikulin involution.
\end{proposition}
\begin{proof}
If $Y=X/\iota$, then we show that $\widetilde{Y}$ is a K3 surface. Let $X'\rightarrow X$ be the blow up at the eight fixed points of $\iota$ and denote by $E_1,\ldots,E_8$ the corresponding exceptional divisors. Then we have the following commutative diagram:
\begin{center}
\begin{tikzpicture}[>=angle 90]
\matrix(a)[matrix of math nodes,
row sep=2em, column sep=2em,
text height=1.5ex, text depth=0.25ex]
{X'&\widetilde{Y}\\
X&Y,\\};
\path[->] (a-1-1) edge node[above]{$\pi$}(a-1-2);
\path[->] (a-1-1) edge node[left]{}(a-2-1);
\path[->] (a-1-2) edge node[left]{}(a-2-2);
\path[->] (a-2-1) edge node[left]{}(a-2-2);
\end{tikzpicture}
\end{center}
where $\pi$ is the $\mathbb{Z}_2$-cover ramified at $E_1+\ldots+E_8$. Therefore,
\begin{equation*}
E_1+\ldots+E_8=K_{X'}\sim\pi^*(K_{\widetilde{Y}})+E_1+\ldots+E_8\implies\pi^*(K_{\widetilde{Y}})\sim0.
\end{equation*}
We can conclude that $2K_{\widetilde{Y}}\sim\pi_*\pi^*(K_{\widetilde{Y}})\sim0$.

Since $\chi(\mathcal{O}_{X'})=\chi(\mathcal{O}_X)=2$, Noether's formula gives that the topological Euler characteristic $\chi_{\topo}(X')$ of $X'$ equals $32$. From this we can argue that $\chi_{\topo}(\widetilde{Y})=24$. Noether's formula again applied to $\widetilde{Y}$ implies that $\chi(\mathcal{O}_{\widetilde{Y}})=2$.

If $B$ is the branch locus of $\pi$, then $\pi_*\mathcal{O}_{X'}=\mathcal{O}_{\widetilde{Y}}\oplus\mathcal{O}_{\widetilde{Y}}(-B/2)$. This guarantees that $q(\widetilde{Y})=0$, and hence $p_g(\widetilde{Y})=1$ because $\chi(\mathcal{O}_{\widetilde{Y}})=2$.

Summing up, $2K_{\widetilde{Y}}\sim0,q(\widetilde{Y})=0$, and $p_g(\widetilde{Y})=1$, which is enough to show that $\widetilde{Y}$ is a K3 surface.
\end{proof}


\subsection{Even eights on a K3 surface}
\begin{definition}
Let $R_1,\ldots,R_m$ be smooth disjoint rational curves on a K3 surface $X$. Then $\{R_1,\ldots,R_m\}$ is called an \emph{even set} if $R_1+\ldots+R_m$ is divisible by $2$ in $\NS(X)$.
\end{definition}
The following result of Nikulin is a combination of \cite[Lemma 3]{nikulin75} and \cite[Corollary 1]{nikulin75}.
\begin{theorem}{\cite{nikulin75}}
\label{theoremaboutcardinalitiesofevensets}
Let $\{R_1,\ldots,R_m\}$ with $m\geq1$ be an even set on a K3 surface. Then $m=8$ or $16$.
\end{theorem}
\begin{observation}
\label{keyfact}
In Theorem~\ref{theoremaboutcardinalitiesofevensets}, it is simple to show that an even set on a K3 surface $X$ cannot have cardinality $4$ (this is actually what we need in the current paper). By contradiction, let $\{C_1,C_2,C_3,C_4\}$ be an even set on $X$ and let $C=\frac{1}{2}(C_1+C_2+C_3+C_4)$. Using the Riemann-Roch theorem, we know that $C$ or $-C$ is effective because $C^2=-2$. Therefore $C$ is effective, because the ample class on $X$ intersects positively the curves $C_i$. But $C_i\subset\Supp(C)$ for all $i\in\{1,2,3,4\}$ because $C_i\cdot C<0$. This implies that $C-C_1-C_2-C_3-C_4=-C$ is also effective. But then $C=0$, which cannot be.
\end{observation}
\begin{definition}
An \emph{even eight} on a K3 surface is an even set of cardinality $8$.
\end{definition}
\begin{example}
\label{exampleofeveneight}
In the proof of Proposition~\ref{involutionwithexactlyeightisolatedfixedpointsissymplectic}, the irreducible curves in the branch locus of the double cover $\pi\colon X'\rightarrow\widetilde{Y}$ form an even eight because they are eight disjoint smooth rational curves and their sum is divisible by $2$ in $\NS(\widetilde{Y})$.
\end{example}


\section{The $4$-dimensional subfamily of K3 surfaces with $\mathbb{Z}_2^2$ symplectic action}
\label{triple-doublek3surfaces}


\subsection{Triple-double K3 surfaces}
\label{definitiontriple-doublek3surfaces}
Consider six lines $\ell_0,\ldots,\ell_5$ in $\mathbb{P}^2$ without triple intersection points. Divide the six lines into three pairs $(\ell_0,\ell_1),(\ell_2,\ell_3),(\ell_4,\ell_5)$. There exists the double cover $X_1\rightarrow\mathbb{P}^2$ branched along $\ell_0+\ell_1$ because $\ell_0+\ell_1\in2\Pic(\mathbb{P}^2)$. The pullback to $X_1$ of $\ell_2+\ell_3$ is also divisible by $2$ in $\Pic(X_1)$, so we have another double cover $X_2\rightarrow X_1$. Repeat this construction on $X_2$ with respect to the last pair of lines to obtain a triple-double cover $X_3\rightarrow\mathbb{P}^2$ branched along $\sum_{i=0}^5\ell_i$. The preimage of $\ell_0\cap\ell_1$ under $X_3\rightarrow\mathbb{P}^2$ consists of four $A_1$ singularities, and the same applies to $\ell_2\cap\ell_3$ and $\ell_4\cap\ell_5$. $X_3$ is smooth away from these twelve singular points. Let $\sigma\colon X\rightarrow X_3$ be the minimal resolution of $X_3$, which is obtained by blowing up the twelve $A_1$ singularities.
\begin{proposition}
\label{prooftripledoublecoverisak3}
With the notation above, $X$ is a K3 surface.
\end{proposition}
\begin{proof}
A projective realization of $X_3$ can be constructed as follows. Let $[W_0:W_1:W_2]$ be coordinates in $\mathbb{P}^2$ and let $L_i=L_i(W_0,W_1,W_2)=0$ be the equation of the line $\ell_i$. Then $X_3\subset\mathbb{P}^5$ is given by the vanishing of the following three quadrics:
\begin{displaymath}
\left\{ \begin{array}{ll}
W_3^2=L_0L_1,\\
W_4^2=L_2L_3,\\
W_5^2=L_4L_5.
\end{array} \right.
\end{displaymath}
From this follows immediately that $K_{X_3}\sim0$ and $h^1(X_3,\mathcal{O}_{X_3})=0$. We conclude that $X$ is a K3 surface because $\sigma^*K_{X_3}\sim K_X$ and $\sigma_*\mathcal{O}_X\cong\mathcal{O}_{X_3}$.
\end{proof}
\begin{definition}
We call a K3 surface $X$ as in Proposition~\ref{prooftripledoublecoverisak3} a \emph{triple-double} K3 surface.
\end{definition}
\begin{observation}
\label{otherwaystoseetriple-doublek3surfaces}
Let $X$ be the triple-double K3 surface as defined above. Let $\Bl_3\mathbb{P}^2$ be the blow up of $\mathbb{P}^2$ at $\ell_0\cap\ell_1,\ell_2\cap\ell_3,\ell_4\cap\ell_5$. Then $X$ can be viewed as a $\mathbb{Z}_2^3$-cover of $\Bl_3\mathbb{P}^2$ branched along the strict preimages $\widehat{\ell}_0,\ldots,\widehat{\ell}_5$ of the six lines in $\mathbb{P}^2$ (see Figure~\ref{threepairsoflines}). More precisely, we first take the double cover of $\Bl_3\mathbb{P}^2$ branched along $\widehat{\ell}_0+\widehat{\ell}_1$, and so on. In this way the $\mathbb{Z}_2^3$-cover is automatically smooth. We can also realize $X$ as a hypersurface in $(\mathbb{P}^1)^3$ as follows. The blow up $\Bl_3\mathbb{P}^2$ can be embedded in $(\mathbb{P}^1)^3$ as a hypersurface of equation given by
\begin{equation*}
\sum_{i,j,k=0,1}c_{ijk}X_iY_jZ_k=0,
\end{equation*}
where the coefficients $c_{ijk}$ are general, nonzero, and the lines $\widehat{\ell}_0,\ldots,\widehat{\ell}_5$ are given by the restriction of the toric boundary of $(\mathbb{P}^1)^3$ (see \cite[Proposition 3.14]{schaffler}). Then $X$ is the following hypersurface in $(\mathbb{P}^1)^3$:
\begin{equation}
\label{222hypersurface}
\sum_{i,j,k=0,1}c_{ijk}X_i^2Y_j^2Z_k^2=0,
\end{equation}
where the $\mathbb{Z}_2^3$-covering map is given by the restriction to $X$ of
\begin{equation*}
([X_0:X_1],[Y_0:Y_1],[Z_0:Z_1])\mapsto([X_0^2:X_1^2],[Y_0^2:Y_1^2],[Z_0^2:Z_1^2]).
\end{equation*}
At this point we have several equivalent descriptions of triple-double K3 surfaces. In what follows, we switch from one perspective to another depending on which one is more convenient for our purposes.
\end{observation}
\begin{figure}
\centering
\includegraphics[scale=0.70,valign=t]{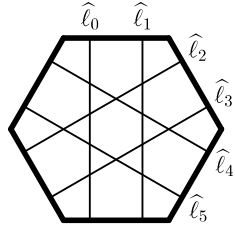}
\caption{Toric picture of $\Bl_3\mathbb{P}^2$ together with the divisor $\widehat{\ell}_0+\ldots+\widehat{\ell}_5$}
\label{threepairsoflines}
\end{figure}
\begin{definition}
\label{familyoftriple-doublek3surfaces}
Consider the realization of a triple-double K3 surface $X$ as a hypersurface in $(\mathbb{P}^1)^3$ given by equation (\ref{222hypersurface}) in Observation~\ref{otherwaystoseetriple-doublek3surfaces}. Up to rescaling the coefficients and up to the torus action on the variables, we can assume that $c_{000},c_{100},c_{010},c_{001}=1$. Let $\mathcal{U}\subset\mathbb{G}_m^4$ be the dense open subset consisting of quadruples $(c_{110},c_{101},c_{011},c_{111})$ which give a triple-double K3 surface. Thinking of $c_{110},c_{101},c_{011},c_{111}$ as variables, let $\mathfrak{X}$ be the subvariety of $\mathcal{U}\times(\mathbb{P}^1)^3$ given by
\begin{gather*}
X_0^2Y_0^2Z_0^2+X_1^2Y_0^2Z_0^2+X_0^2Y_1^2Z_0^2+X_0^2Y_0^2Z_1^2\\
+c_{110}X_1^2Y_1^2Z_0^2+c_{101}X_1^2Y_0^2Z_1^2+c_{011}X_0^2Y_1^2Z_1^2+c_{111}X_1^2Y_1^2Z_1^2=0.
\end{gather*}
We refer to $\mathfrak{X}\rightarrow\mathcal{U}$ as the \emph{family} of triple-double K3 surfaces.
\end{definition}


\subsection{Involutions of a triple-double K3 surface}
\begin{proposition}
\label{involutiontriple-doublek3surface}
Let $X$ be a triple-double K3 surface, and view it as a hypersurface in $(\mathbb{P}^1)^3$ of equation
\begin{equation*}
\sum_{i,j,k=0,1}c_{ijk}X_i^2Y_j^2Z_k^2=0,
\end{equation*}
as explained in Observation~\ref{otherwaystoseetriple-doublek3surfaces}. Denote by $\iota_{ijk}$ the restriction to $X$ of
{\footnotesize
\begin{align*}
(\mathbb{P}^1)^3&\rightarrow(\mathbb{P}^1)^3,\\
([X_0:X_1],[Y_0:Y_1],[Z_0:Z_1])&\mapsto([X_0:(-1)^iX_1],[Y_0:(-1)^jY_1],[Z_0:(-1)^kZ_1]).
\end{align*}
}%
Then the following hold:
\begin{itemize}
\item[(i)] $\iota_{111}$ is an Enriques involution (i.e., $X/\iota_{111}$ is an Enriques surface);
\item[(ii)] If $\iota\in\{\iota_{100},\iota_{010},\iota_{001}\}$, then $\iota$ is a non-symplectic involution and $X/\iota$ is a smooth rational surface. The fixed points locus of $\iota$ consists of two disjoint smooth genus $1$ curves (and hence $\iota$ is of parabolic type according to \cite[Section 2.8]{alexeevnikulin});
\item[(iii)] $\iota_{110},\iota_{101},\iota_{011}$ are Nikulin involutions with pairwise disjoint sets of fixed points.
\end{itemize}
In particular, the group of automorphisms $\{\id_X,\iota_{110},\iota_{101},\iota_{011}\}\cong\mathbb{Z}_2^2$ acts symplectically on $X$.
\end{proposition}
\begin{proof}
$\iota_{111}$ is an Enriques involution because it has no fixed points (recall that the coefficients $c_{ijk}$ are nonzero). $\iota_{110},\iota_{101},\iota_{011}$ are symplectic involutions because each one has exactly eight fixed points (see Proposition~\ref{involutionwithexactlyeightisolatedfixedpointsissymplectic}). It is simple to observe that $\iota_{110},\iota_{101},\iota_{011}$ have pairwise disjoint sets of fixed points.

Let us prove (ii) for $\iota_{100}$ (similar arguments hold for $\iota_{010},\iota_{001}$). The fixed points locus of $\iota_{100}$ is given by
\begin{equation*}
\left\{X_0=0,\sum_{j,k=0,1}c_{1jk}Y_j^2Z_k^2=0\right\}\amalg\left\{X_1=0,\sum_{j,k=0,1}c_{0jk}Y_j^2Z_k^2=0\right\},
\end{equation*}
which are two disjoint smooth genus $1$ curves. This is enough to show that $X/\iota_{100}$ is a smooth rational surface (see \cite[Chapter 2]{alexeevnikulin}).
\end{proof}
\begin{observation}
\label{triple-doublek3sforma4dimfamily}
Let $\mathfrak{X}\rightarrow\mathcal{U}$ be the family of triple-double K3 surfaces in Definition~\ref{familyoftriple-doublek3surfaces}, which has a $4$-dimensional parameter space $\mathcal{U}$. Denote by $X_u$ the fiber over a point $u\in\mathcal{U}$. Let us show that $\mathfrak{X}\rightarrow\mathcal{U}$ is a \emph{$4$-dimensional family} of K3 surfaces, and by this we mean that there is no positive dimensional subvariety $C\subset\mathcal{U}$ such that $X_u\cong X_v$ for all $u,v\in C$. Assume by contradiction that there exists such $C\subset\mathcal{U}$. Let $G$ be the group generated by the involutions $\{\iota_{110},\iota_{101},\iota_{011}\}$ in Proposition~\ref{involutiontriple-doublek3surface}. Then $G$ acts on the whole space $\mathfrak{X}$ inducing a symplectic action on each fiber of $\mathfrak{X}\rightarrow\mathcal{U}$. If $\mathcal{Z}$ denotes the minimal resolution of the quotient of $\mathfrak{X}$ by $G$, then $\mathcal{Z}\rightarrow\mathcal{U}$ is the family of K3 surfaces given by the minimal resolutions of the double covers of $\mathbb{P}^2$ branched along six lines without triple intersection points (this aspect is discussed more in Section~\ref{three4dimensionalfamiliesofk3surfacesrelated}). We have that $\mathcal{Z}\rightarrow\mathcal{U}$ is a $4$-dimensional family of K3 surfaces (see \cite[Remark 5.9]{kloosterman}). But then the fibers of $\mathcal{Z}\rightarrow\mathcal{U}$ over the positive dimensional subvariety $C\subset\mathcal{U}$ would be isomorphic to each other, which cannot be.
\end{observation}
\begin{observation}
\label{infiniteautomorphismgroupandinfinitelymanysmoothrationalcurves}
A triple-double K3 surface $X$ has infinite automorphism group because it covers an Enriques surface by Proposition~\ref{involutiontriple-doublek3surface}(i) (see \cite[Page 194]{kondo86}). Moreover, $X$ contains smooth rational curves (see Proposition~\ref{configurationof24smoothrationalcurves} for a more detailed discussion). Therefore, $X$ contains infinitely many smooth rational curves by \cite[Corollary 4.7]{huybrechts}.
\end{observation}


\subsection{A configuration of $24$ smooth rational curves on a triple-double K3 surface}
\begin{proposition}
\label{configurationof24smoothrationalcurves}
Let $X$ be a triple-double K3 surface and let $X\rightarrow\Bl_3\mathbb{P}^2$ be the corresponding $\mathbb{Z}_2^3$-cover as described in Observation~\ref{otherwaystoseetriple-doublek3surfaces}. Then the preimage of the six $(-1)$-curves on $\Bl_3\mathbb{P}^2$ under the covering map gives a configuration of $24$ smooth rational curves on $X$ whose dual graph is shown in Figure~\ref{dualgraph24smoothrationalcurves}.
\end{proposition}
\begin{proof}
Split the $\mathbb{Z}_2^3$-cover $X\rightarrow\Bl_3\mathbb{P}^2$ into three double covers, each one branched along a pair of curves. Then the configuration of $24$ smooth rational curves on $X$ can be computed by taking appropriate branched double covers starting from the six $(-1)$-curves on $\Bl_3\mathbb{P}^2$ as it is shown in Figure~\ref{firsttwodoublecoverscurve}. The result of the last double cover is shown in Figure~\ref{dualgraph24smoothrationalcurves}.
\end{proof}
\begin{figure}
\centering
\includegraphics[scale=0.55,valign=t]{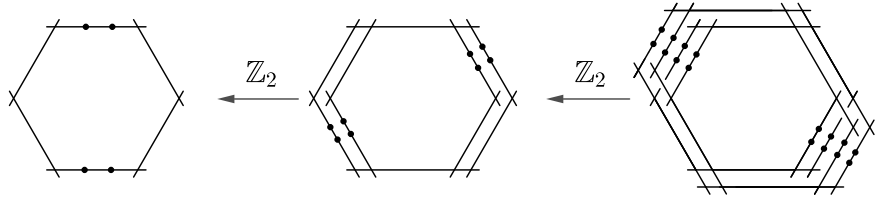}
\caption{First two double covers of the six $(-1)$-curves of $\Bl_3\mathbb{P}^2$. The marked points represent the branch locus}
\label{firsttwodoublecoverscurve}
\end{figure}
\begin{figure}
\centering
\includegraphics[scale=0.50,valign=t]{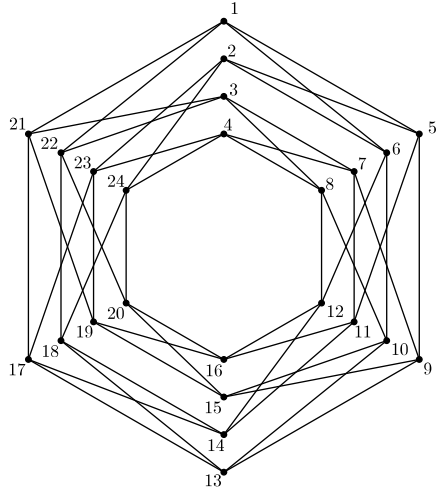}
\caption{Dual graph of the smooth rational curves on a triple-double K3 surface arising as the $\mathbb{Z}_2^3$-cover of the six $(-1)$-curves of $\Bl_3\mathbb{P}^2$}
\label{dualgraph24smoothrationalcurves}
\end{figure}
\begin{definition}
\label{notationfor24curvesonX}
Let $X$ be a triple-double K3 surface. Denote by $R_1,\ldots,R_{24}$ the $24$ smooth rational curves on $X$ described in Proposition~\ref{configurationof24smoothrationalcurves}, and label them as shown in Figure~\ref{dualgraph24smoothrationalcurves}.
\end{definition}


\section{The N\'eron-Severi lattice of a triple-double K3 surface with minimal Picard number}
\label{theneronseverilatticeofatriple-doublek3surface}
Let $X$ be a triple-double K3 surface with minimal Picard number. In this section, we show that $\NS(X)$ is generated by the $24$ smooth rational curves $R_1,\ldots,R_{24}$. Moreover, we give an explicit $\mathbb{Z}$-basis for $\NS(X)$ which decomposes $\NS(X)$ as $U\oplus E_8\oplus Q$, where $Q$ is described explicitly in Lemma~\ref{thesublatticesSandQ}(iii).


\subsection{The sublattices $S,Q\subset\NS(X)$}
\begin{lemma}
\label{thesublatticesSandQ}
Let $X$ be a triple-double K3 surface. Let $S$ be the sublattice of $\NS(X)$ generated by
\begin{equation*}
\mathcal{S}=\{R_1,~R_5,~R_9,~R_{13},~R_{17},~R_{23},~R_4,~R_{15},~R_8,~R_3\}.
\end{equation*}
Let $Q$ be the sublattice of $\NS(X)$ generated by
\begin{gather*}
\mathcal{Q}=\{R_{16},~R_{14}-R_{21}+R_{22},~R_{11}-R_2+R_{19}-R_{20},\\
R_{17}+2R_{14}-R_{18}-R_{19}+R_{20},~R_{12}-R_{10}+R_{18}+R_{20},\\
R_3+2R_{22}-2R_6-R_{12}\}.
\end{gather*}
Then the following hold:
\begin{itemize}
\item[(i)] $\mathcal{S}$ is a $\mathbb{Z}$-basis for $S$ and $S\cong U\oplus E_8$;
\item[(ii)] $Q\subset S^\perp$ in $\NS(X)$;
\item[(iii)] $\mathcal{Q}$ is a $\mathbb{Z}$-basis for $Q$ and its corresponding Gram matrix is given by
\begin{displaymath}
\left( \begin{array}{cccccc}
-2&0&1&0&2&-1\\
0&-6&-1&-4&4&-5\\
1&-1&-8&6&2&0\\
0&-4&6&-16&4&-2\\
2&4&2&4&-8&6\\
-1&-5&0&-2&6&-12
\end{array} \right).
\end{displaymath}
\end{itemize}
\end{lemma}
\begin{proof}
All the statements above can be checked explicitly. We only remark that $S$ is isometric to $U\oplus E_8$ because $S$ is an even unimodular lattice of signature $(1,9)$ (see Figure~\ref{subgraphextendede8plussection}).
\end{proof}
\begin{figure}
\centering
\includegraphics[scale=0.50,valign=t]{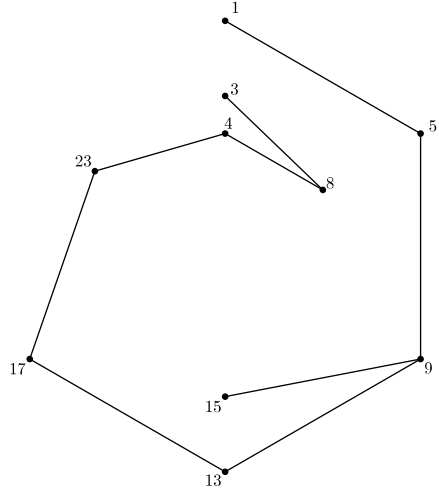}
\caption{Subgraph of the dual graph of the curves $R_1,\ldots,R_{24}$ with vertices the $\mathbb{Z}$-basis of the lattice $S$ in Lemma~\ref{thesublatticesSandQ}. These constitute a type $II^*$ singular fiber of an elliptic pencil (see \cite[Figure 1]{kodaira}) together with the section $R_3$}
\label{subgraphextendede8plussection}
\end{figure}
\begin{corollary}
\label{importantpropositionforequalitiesinNS}
Let $X$ be a triple-double K3 surface with minimal Picard number. Then $\rho(X)=16$ and $\NS(X)$ is generated over $\mathbb{Q}$ by $R_1,\ldots,R_{24}$. In particular, given $D_1,D_2\in\NS(X)$, we have that $D_1=D_2$ if and only if $D_1\cdot R_i=D_2\cdot R_i$ for all $i\in\{1,\ldots,24\}$.
\end{corollary}
\begin{proof}
If $X$ is any triple-double K3 surface, then the sublattice $S\oplus Q\subseteq\NS(X)$ has rank $16$, hence $\rho(X)\geq\rk(S\oplus Q)=16$. To show that the minimum for $\rho(X)$ equals $16$, consider the $4$-dimensional coarse moduli space $\mathcal{M}$ of $(S\oplus Q)$-polarized K3 surfaces (see \cite{dolgachev}). Let $\mathfrak{X}\rightarrow\mathcal{U}$ be the family of triple-double K3 surfaces constructed in Definition~\ref{familyoftriple-doublek3surfaces}. Then the induced morphism $p\colon\mathcal{U}\rightarrow\mathcal{M}$ is dominant because $\mathfrak{X}\rightarrow\mathcal{U}$ is a $4$-dimensional family by Observation~\ref{triple-doublek3sforma4dimfamily}. This implies that the minimal Picard number for $X$ is $16$.

Let $L\subseteq\NS(X)$ be the sublattice of $\NS(X)$ generated by $R_1,\ldots,R_{24}$. Then $S\oplus Q\subseteq L\subseteq\NS(X)$, which implies that $L$ generates $\NS(X)$ over $\mathbb{Q}$ if $\rho(X)=16$.

The last statement about $D_1,D_2$ follows from this and from the fact that on a K3 surface numerical equivalence of divisors coincides with linear equivalence (see Section~\ref{k3surfacesbasics}).
\end{proof}
\begin{remark}
Consider the family of triple-double K3 surfaces $\mathfrak{X}\rightarrow\mathcal{U}$ in Definition~\ref{familyoftriple-doublek3surfaces}. Denote by $X_u$ the fiber over a point $u\in\mathcal{U}$. Then by \cite[Chapter 6, \S2.5]{huybrechts} the K3 surface $X_u$ has minimal Picard number for a \emph{very general} $u\in\mathcal{U}$ (i.e., for all $u\in\mathcal{U}\setminus Z$, where $Z$ is the union of countably many Zariski closed subsets of $\mathcal{U}$).
\end{remark}


\subsection{The discriminant group of $Q$ and proof that $S\oplus Q=\NS(X)$}
\begin{lemma}
\label{explicitdiscriminantgroupofQ}
The discriminant group of $Q$ (and hence the discriminant group of $S\oplus Q$) is isomorphic to $\mathbb{Z}_2^2\oplus\mathbb{Z}_4^2$, and it is generated by the classes modulo $Q$ of the following elements in $Q^*$:
\begin{gather*}
v_1=\left(\frac{1}{2},-\frac{1}{2},0,0,\frac{1}{2},0\right),~v_2=\left(-\frac{1}{2},\frac{1}{2},0,0,0,0\right),\\
w_1=\left(\frac{1}{2},0,0,-\frac{1}{4},0,0\right),~w_2=\left(0,0,\frac{1}{4},-\frac{1}{4},-\frac{1}{4},-\frac{1}{4}\right),
\end{gather*}
where the coordinates are with respect to the $\mathbb{Z}$-basis $\mathcal{Q}$ (see Lemma~\ref{thesublatticesSandQ}).
\end{lemma}
\begin{proof}
Let $B$ be the Gram matrix of $Q$ associated to the $\mathbb{Z}$-basis $\mathcal{Q}$ in Lemma~\ref{thesublatticesSandQ}(iii). Then the lattice $Q^*$ is generated over $\mathbb{Z}$ by the rows of $B^{-1}$, and to understand the discriminant group of $Q$ we compute the Smith normal form of $B^{-1}$. This can be done using Sage (see \cite[function smith\underline{\hspace{.25cm}}form()]{stein}), which gives us $M_1,M_2\in\SL_{6}(\mathbb{Z})$ such that $M_1B^{-1}M_2$ is the diagonal matrix $\diag\left(1,1,\frac{1}{2},\frac{1}{2},\frac{1}{4},\frac{1}{4}\right)$. This implies that $A_Q\cong\mathbb{Z}_2^2\oplus\mathbb{Z}_4^2$. Moreover, the rows of $M_1B^{-1}$ give us an alternative $\mathbb{Z}$-basis of $Q^*$ to work with. One has
\begin{displaymath}
M_1B^{-1}=\left( \begin{array}{cccccc}
0&0&0&0&0&1\\
1&0&0&0&1&0\\
\frac{1}{2}&-\frac{1}{2}&0&0&\frac{1}{2}&0\\
-\frac{1}{2}&\frac{1}{2}&0&0&0&0\\
\frac{1}{2}&0&0&-\frac{1}{4}&0&0\\
0&0&\frac{1}{4}&-\frac{1}{4}&-\frac{1}{4}&-\frac{1}{4}
\end{array} \right).
\end{displaymath}
The first $2$ rows represent elements in $Q$ and the last $4$ are generators of the discriminant group $A_Q$, which we denote by $v_1,v_2,w_1,w_2$ respectively.
\end{proof}
\begin{lemma}
\label{isotropicvectorsnotinNSX}
Let $v_1,v_2,w_1,w_2\in Q^*$ as in Lemma~\ref{explicitdiscriminantgroupofQ}. Then the values of the symmetric bilinear form $b_{Q^*}$ evaluated at pairs of these vectors are shown in the following table.
\begin{displaymath}
\begin{array}{c|cccc}
b_{Q^*}&v_1&v_2&w_1&w_2\\
\hline\\[-0.4cm]
v_1&-5&\frac{5}{2}&-1&-\frac{1}{2}\\[.1cm]
v_2&\frac{5}{2}&-2&1&\frac{1}{2}\\[.1cm]
w_1&-1&1&-\frac{3}{2}&-\frac{5}{4}\\[.1cm]
w_2&-\frac{1}{2}&\frac{1}{2}&-\frac{5}{4}&-\frac{11}{4}
\end{array}
\end{displaymath}
In particular, the only isotropic elements in the discriminant group $A_Q$ with respect to the discriminant quadratic form $q_Q\colon A_Q\rightarrow\mathbb{Q}/2\mathbb{Z}$ are the classes of
\begin{equation*}
2w_1,~v_2,~v_2+2w_1,~v_1+2w_2,~v_1+2w_1+2w_2,~v_1+v_2,~v_1+v_2+2w_1.
\end{equation*}
Moreover, these are not contained in $\NS(X)$.
\end{lemma}
\begin{proof}
A direct calculation gives the values in the table. Recall that the coordinates of the vectors $v_1,v_2,w_1,w_2$ are with respect to the $\mathbb{Z}$-basis $\mathcal{Q}$ in Lemma~\ref{explicitdiscriminantgroupofQ}. So, for instance, we have that
\begin{equation*}
v_1=\frac{1}{2}R_{16}-\frac{1}{2}(R_{14}-R_{21}+R_{22})+\frac{1}{2}(R_{12}-R_{10}+R_{18}+R_{20}).
\end{equation*}

The listed isotropic vectors are easily obtained after computing all the possible $(av_1+bv_2+cw_1+dw_2)^2$ with $((a,b),(c,d))\in\{0,1\}^2\times\{0,1,2,3\}^2$.

Given two vectors $u_1,u_2\in Q^*$, we write $u_1\approx u_2$ if $u_1-u_2\in\NS(X)$. Then we have that
\begin{gather*}
2w_1\approx\frac{R_{17}+R_{18}+R_{19}+R_{20}}{2},\\
v_2\approx\frac{R_{14}+R_{16}+R_{21}+R_{22}}{2},\\
v_1+v_2\approx\frac{R_{10}+R_{12}+R_{18}+R_{20}}{2},\\
v_1+v_2+2w_1\approx\frac{R_{10}+R_{12}+R_{17}+R_{19}}{2}.
\end{gather*}
This implies that $2w_1,v_2,v_1+v_2,v_1+v_2+2w_1\notin\NS(X)$ by Observation~\ref{keyfact}.

We are left to show that $v_2+2w_1,v_1+2w_2,v_1+2w_1+2w_2\notin\NS(X)$. We analyze them separately.
\begin{itemize}
\item Assume by contradiction that $v_2+2w_1\in\NS(X)$. Then we have that
\begin{equation*}
v_2+2w_1\approx\frac{R_{14}+R_{16}+R_{17}+R_{18}+R_{19}+R_{20}+R_{21}+R_{22}}{2}=\alpha,
\end{equation*}
and $\alpha\in\NS(X)$. Using Corollary~\ref{importantpropositionforequalitiesinNS}, it is easy to observe that $R_{13}+R_{14}+R_{17}+R_{18}=R_{15}+R_{16}+R_{19}+R_{20}$ because $(R_{13}+R_{14}+R_{17}+R_{18})\cdot R_i=(R_{15}+R_{16}+R_{19}+R_{20})\cdot R_i$ for all $i=1,\ldots,24$. In particular, we have that
\begin{equation*}
\frac{R_{13}+R_{14}+R_{17}+R_{18}+R_{15}+R_{16}+R_{19}+R_{20}}{2}=\alpha_1\in\NS(X).
\end{equation*}
But then $\alpha+\alpha_1\in\NS(X)$, which contradicts Observation~\ref{keyfact} because
\begin{equation*}
\alpha+\alpha_1\approx\frac{R_{13}+R_{15}+R_{21}+R_{22}}{2}.
\end{equation*}
\item Similarly, assume that $v_1+2w_2\in\NS(X)$. Then we have that
\begin{gather*}
v_1+2w_2\\
\approx\frac{R_2+R_3+R_{11}+R_{12}+R_{14}+R_{16}+R_{17}+R_{18}+R_{21}+R_{22}}{2}=\beta,
\end{gather*}
and $\beta\in\NS(X)$. Using Corollary~\ref{importantpropositionforequalitiesinNS}, we can verify that $R_{11}+R_{12}+R_{14}+R_{16}=R_1+R_3+R_{21}+R_{22}$, so that
\begin{equation*}
\frac{R_{11}+R_{12}+R_{14}+R_{16}+R_1+R_3+R_{21}+R_{22}}{2}=\beta_1\in\NS(X).
\end{equation*}
But then $\beta+\beta_1\in\NS(X)$, which contradicts Observation~\ref{keyfact} because
\begin{equation*}
\beta+\beta_1\approx\frac{R_1+R_2+R_{17}+R_{18}}{2}.
\end{equation*}
\item If $v_1+2w_1+2w_2\in\NS(X)$, then
\begin{gather*}
v_1+2w_1+2w_2\\
\approx\frac{R_2+R_3+R_{11}+R_{12}+R_{14}+R_{16}+R_{19}+R_{20}+R_{21}+R_{22}}{2}=\gamma,
\end{gather*}
and $\gamma\in\NS(X)$. Let $\beta_1$ as in the previous point. Then $\gamma+\beta_1\in\NS(X)$, which is not allowed by Observation~\ref{keyfact} because
\begin{equation*}
\gamma+\beta_1\approx\frac{R_1+R_2+R_{19}+R_{20}}{2}.
\end{equation*}
\end{itemize}
\end{proof}
\begin{theorem}
\label{mainresultNSgeneratedby24smoothrationalcurves}
Let $X$ be a triple-double K3 surface with $\rho(X)=16$. Then the following hold:
\begin{itemize}
\item[(i)] $\NS(X)=S\oplus Q\cong U\oplus E_8\oplus Q$;
\item[(ii)] $\NS(X)$ is generated be the smooth rational curves $R_1,\ldots,R_{24}$;
\item[(iii)] The discriminant group of $\NS(X)$ is $\mathbb{Z}_2^2\oplus\mathbb{Z}_4^2$.
\end{itemize}
\end{theorem}
\begin{proof}
If $\rho(X)=16$, then $\NS(X)$ is an even overlattice of $S\oplus Q$. Therefore $\NS(X)/(S\oplus Q)$ corresponds to a subgroup of $A_{(S\oplus Q)}\cong A_Q$ which is isotropic with respect to the discriminant quadratic form $q_Q\colon A_Q\rightarrow\mathbb{Q}/2\mathbb{Z}$ (see Theorem~\ref{overlatticesinbijectionwithisotropicsubgroups}). However, Proposition~\ref{isotropicvectorsnotinNSX} shows that no isotropic vector of $A_Q$ can be contained in $\NS(X)$. This implies that $\NS(X)$ has to be equal to $S\oplus Q$, which recall is isometric to $U\oplus E_8\oplus Q$ by Lemma~\ref{thesublatticesSandQ}(i). It also follows that $\NS(X)$ is generated by $R_1,\ldots,R_{24}$ because the sublattice of $\NS(X)$ generated by these $24$ curves contains $S\oplus Q$. Finally, $A_{\NS(X)}=A_{(S\oplus Q)}\cong A_Q$, which is isomorphic to $\mathbb{Z}_2^2\oplus\mathbb{Z}_4^2$ by Lemma~\ref{explicitdiscriminantgroupofQ}.
\end{proof}
\begin{remark}
\label{exampleofbasiswhichgivessplitting}
If $X$ is a triple-double K3 surface with $\rho(X)=16$, then the following $\mathbb{Z}$-basis of $\NS(X)$ realizes it as a direct sum of $U,E_8$, and $Q$ (see Lemma~\ref{thesublatticesSandQ}):
\begin{gather*}
\NS(X)=S\oplus Q\\
=\langle2R_1+2R_4+4R_5+R_8+6R_9+5R_{13}+3R_{15}+4R_{17}+3R_{23},~R_3\rangle_\mathbb{Z}\\
\oplus\langle R_1,~R_5,~R_9,~R_{13},~R_{17},~R_{23},~R_4,~R_{15}\rangle_\mathbb{Z}\\
\oplus\langle R_{16},~R_{14}-R_{21}+R_{22},~R_{11}-R_2+R_{19}-R_{20},\\
R_{17}+2R_{14}-R_{18}-R_{19}+R_{20},~R_{12}-R_{10}+R_{18}+R_{20},\\
R_3+2R_{22}-2R_6-R_{12}\rangle_\mathbb{Z}\cong U\oplus E_8\oplus Q.
\end{gather*}
\end{remark}
\begin{remark}
For a triple-double K3 surface $X$ with $\rho(X)=16$, a splitting $\NS(X)\cong U\oplus E_8\oplus P$ for some lattice $P$ is predicted abstractly by Theorem~\ref{splittingofevenlattices} as follows. We know that $\NS(X)$ is even, of signature $(1,15)$, and discriminant group $\mathbb{Z}_2^2\oplus\mathbb{Z}_2^4$. It follows from Theorem~\ref{splittingofevenlattices}(i) that $\NS(X)$ is isometric to $E_8\oplus P'$ for some lattice $P'$. But then $P'$ is even, of signature $(1,7)$, and its discriminant group is also $\mathbb{Z}_2^2\oplus\mathbb{Z}_4^2$. In particular, we can apply Theorem~\ref{splittingofevenlattices}(ii) to $P'$ to argue that $P'=U\oplus P$ for some lattice $P$. In conclusion, we have that $\NS(X)\cong U\oplus E_8\oplus P$, but all we know about $P$ is that it is even, of signature $(0,6)$, and discriminant group $\mathbb{Z}_2^2\oplus\mathbb{Z}_4^2$. However, in our case we were able to provide an explicit lattice $Q$ (see Lemma~\ref{thesublatticesSandQ}(iii)) which realizes the splitting $\NS(X)\cong U\oplus E_8\oplus Q$.
\end{remark}


\subsection{The transcendental lattice of a triple-double K3 surface with minimal Picard number}
\begin{proposition}
\label{transcendentallatticetriple-doublek3surface}
Let $X$ be a triple-double K3 surface with $\rho(X)=16$. Then the transcendental lattice $T_X$ is isometric to $U\oplus U(2)\oplus\langle-4\rangle^{\oplus2}$.
\end{proposition}
\begin{proof}
The transcendental lattice $T_X$ is an even lattice of signature $(2,4)$. Let us study its discriminant quadratic form. The table in the statement of Lemma~\ref{isotropicvectorsnotinNSX} gives the discriminant quadratic form of the lattice $\NS(X)=S\oplus Q$ with respect to $v_1,v_2,w_1,w_2$, whose classes modulo $Q$ generate the discriminant group. Changing basis to $v_2,v_1+v_2+2w_1,v_1+2w_1-w_2,v_1+w_1-w_2$, the discriminant quadratic form becomes
\begin{displaymath}
\left( \begin{array}{cccc}
0&-\frac{1}{2}&0&0\\
-\frac{1}{2}&0&0&0\\
0&0&\frac{1}{4}&0\\
0&0&0&\frac{1}{4}
\end{array} \right),
\end{displaymath}
where the entries on the main diagonal (resp. off the main diagonal) are considered modulo $2\mathbb{Z}$ (resp. modulo $\mathbb{Z}$). It follows from Theorem~\ref{sublatticeandperpsamediscriminantgroup} that the discriminant group of $T_X$ is $\mathbb{Z}_2^2\oplus\mathbb{Z}_4^2$ ($\NS(X)$ is a primitive sublattice of the even unimodular lattice $H^2(X;\mathbb{Z})$), and its discriminant quadratic form is the opposite of the matrix above. Observe that the lattice $U\oplus U(2)\oplus\langle-4\rangle^{\oplus2}$ is even, of signature $(2,4)$, and its discriminant quadratic form equals the discriminant quadratic form of $T_X$. We can conclude by Theorem~\ref{uniquenessofevenlattices} that $T_X$ is isometric to $U\oplus U(2)\oplus\langle-4\rangle^{\oplus2}$.
\end{proof}
Recall that projective Kummer surfaces are special types of projective K3 surfaces obtained as the minimal resolution of the quotient of an abelian surface $A$ by the inversion morphism $a\mapsto-a$. We denote such Kummer surface by $\Km(A)$. A celebrated example of Kummer surface is $\Km(E_i\times E_i)$, where $E_i$ is the elliptic curve $\mathbb{C}/(\mathbb{Z}\oplus i\mathbb{Z})$. The automorphism group of $\Km(E_i\times E_i)$ is studied in \cite{keumkondo}, and in \cite{garbagnatisarti09} this Kummer surface was used to compute $\Omega_G$ and $\Omega_G^\perp$ for $G=\mathbb{Z}_2\oplus\mathbb{Z}_4,\mathbb{Z}_2^2,\mathbb{Z}_4$.

\begin{observation}
The transcendental lattice of $\Km(E_i\times E_i)$ is isometric to $\langle4\rangle^{\oplus2}$ (see \cite[Section 1]{keumkondo}), and $\langle4\rangle^{\oplus2}$ admits a primitive embedding into the transcendental lattice of a triple-double K3 surface with minimal Picard number, which is $U\oplus U(2)\oplus\langle-4\rangle^{\oplus2}$ by Proposition~\ref{transcendentallatticetriple-doublek3surface}. To show this, consider the sublattice $L\subset U\oplus U(2)\oplus\langle-4\rangle^{\oplus2}$ generated by $\alpha=(1,2,0,0,0,0),\beta=(0,0,1,1,0,0)$. We have that $L$ is isometric to $\langle4\rangle^{\oplus2}$ because $\alpha^2=\beta^2=4$ and $\alpha\cdot\beta=0$. To prove that $L$ is a primitive sublattice, assume we have $m(x,y,z,w,u,v)\in L$ for some integer $m>1$ and $(x,y,z,w,u,v)\in U\oplus U(2)\oplus\langle-4\rangle^{\oplus2}$. Then $y=2x,z=w,u=v=0$, which implies that $(x,y,z,w,u,v)\in L$.

Therefore, one can ask if the Kummer surface $\Km(E_i\times E_i)$ is a specialization of the family of triple-double K3 surfaces. If this is true, then there should be a configuration of three pairs of lines in $\mathbb{P}^2$ such that the minimal resolution of the appropriate $\mathbb{Z}_2^3$-cover of $\mathbb{P}^2$ gives $\Km(E_i\times E_i)$ (by appropriate we mean the usual chain of three double covers). The next theorem describes explicitly this line arrangement.
\end{observation}
\begin{theorem}
\label{explicitlinearrangementthatgiveskmeixei}
Consider three pairs of lines $(\ell_0,\ell_1),(\ell_2,\ell_3),(\ell_4,\ell_5)$ in $\mathbb{P}^2$ such that the resulting line arrangement has exactly four triple intersection points as shown in Figure~\ref{linearrangementKm(EixEi)} (this is unique up to an automorphism of $\mathbb{P}^2$). Then the minimal resolution of the appropriate $\mathbb{Z}_2^3$-cover of $\mathbb{P}^2$ branched along these three pairs is isomorphic to the singular Kummer surface $\Km(E_i\times E_i)$.
\end{theorem}
\begin{proof}
Let $\Bl_3\mathbb{P}^2$ be the blow up of $\mathbb{P}^2$ at the three marked points in Figure~\ref{linearrangementKm(EixEi)}. Let $X\rightarrow\Bl_3\mathbb{P}^2$ be the appropriate $\mathbb{Z}_2^3$-cover branched along the three pairs of lines $(\widehat{\ell}_0,\widehat{\ell}_1),(\widehat{\ell}_2,\widehat{\ell}_3),(\widehat{\ell}_4,\widehat{\ell}_5)$. In particular, $X$ has exactly four $A_1$ singularities, one over each triple intersection point of the line arrangement. Following Observation~\ref{otherwaystoseetriple-doublek3surfaces}, the surface $X$ can be viewed as the following hypersurface in $(\mathbb{P}^1)^3$:
\begin{equation*}
X_1^2Y_0^2Z_0^2+X_0^2Y_1^2Z_0^2+X_0^2Y_0^2Z_1^2+X_1^2Y_1^2Z_1^2=0.
\end{equation*}
The blow up of $X$ at the four $A_1$ singularities is a K3 surface.

Consider the genus $1$ fibration on $X$ given by the restriction to $X$ of the projection $\pi_3\colon([X_0:X_1],[Y_0:Y_1],[Z_0:Z_1])\mapsto[Z_0:Z_1]$. The general fiber of this fibration is a genus $1$ curve in $\mathbb{P}^1\times\mathbb{P}^1$ given by
\begin{equation*}
C\colon \lambda^2X_1^2Y_0^2+\lambda^2X_0^2Y_1^2+\mu^2X_0^2Y_0^2+\mu^2X_1^2Y_1^2=0,
\end{equation*}
for a general $[\lambda:\mu]\in\mathbb{P}^1$. The restriction to $C$ of the projection $([X_0:X_1],[Y_0:Y_1])\mapsto[Y_0:Y_1]$ realizes $C$ as a double cover of $\mathbb{P}^1$ branched along $[i\lambda:\mu],[-i\lambda:\mu],[i\mu:\lambda],[-i\mu:\lambda]$. If we set $\sigma=i(\lambda/\mu)$ in the affine patch of $\mathbb{P}^1$ where $Y_1\neq0$, then the four branch points above become respectively
\begin{equation*}
\sigma,~-\sigma,~-\frac{1}{\sigma},~\frac{1}{\sigma}.
\end{equation*}
Using the automorphism of $\mathbb{P}^1$ given by $z\mapsto(\frac{\sigma+\sigma^3}{2})\cdot\frac{z-\sigma}{\sigma z-1}$, we can move these branch points to
\begin{equation*}
0,~\sigma^2,~\frac{(1+\sigma^2)^2}{4},~\infty,
\end{equation*}
respectively. But then the elliptic fibration $\pi_3|_X\colon X\rightarrow\mathbb{P}^1$ is isomorphic to the elliptic fibration (6) in \cite[Section 4]{garbagnatisarti09} (set $\tau=1$), which gives $\Km(E_i\times E_i)$.
\end{proof}
\begin{figure}
\centering
\includegraphics[scale=0.60,valign=t]{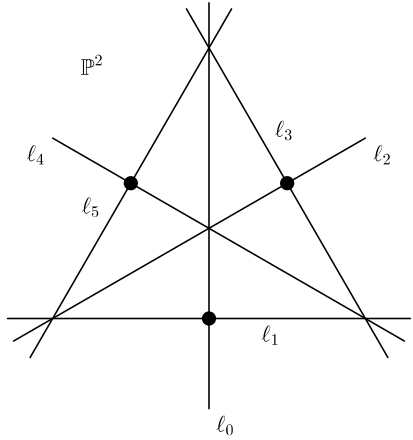}
\caption{The three pairs of lines $(\ell_0,\ell_1),(\ell_2,\ell_3),(\ell_4,\ell_5)$ in $\mathbb{P}^2$ such that the minimal resolution of the appropriate $\mathbb{Z}_2^3$-cover of $\mathbb{P}^2$ branched along these three pairs gives $\Km(E_i\times E_i)$}
\label{linearrangementKm(EixEi)}
\end{figure}


\subsection{Further properties of triple-double K3 surfaces}
\label{furtherpropertiesofthelatticensx}
The next proposition, among other things, shows that triple-double K3 surfaces with minimal Picard number are disjoint from the family of K3 surfaces with $\mathbb{Z}_2^3$ symplectic action.
\begin{proposition}
\label{XnoZ24andnoZ23symplecticactions}
Let $X$ be a triple-double K3 surface with $\rho(X)=16$. Then the following hold:
\begin{itemize}
\item[(i)] $X$ does not admit $\mathbb{Z}_2^3$ symplectic action. In particular, $X$ does not admit $\mathbb{Z}_2^4$ symplectic action, hence \cite[Proposition 6.2]{garbagnatisarti09} cannot be used to compute $\NS(X)$;
\item[(ii)] $X$ is not isomorphic to the minimal resolution of the quotient of a K3 surface by a symplectic action of the group $\mathbb{Z}_2^3$. In particular, $X$ is not isomorphic to the minimal resolution of the quotient of a K3 surface by a symplectic action of the group $\mathbb{Z}_2^4$, hence $\NS(X)$ cannot be computed using \cite[Theorem 8.3]{garbagnatisarti16}.
\end{itemize}
\end{proposition}
\begin{proof}
To prove part (i), $X$ admits $\mathbb{Z}_2^3$ symplectic action if and only if the transcendental lattice $T_X$ embeds primitively into $\Omega_{\mathbb{Z}_2^3}^\perp$ (see \cite{nikulin79}). But $T_X\cong U\oplus U(2)\oplus\langle-4\rangle^{\oplus2}$ by Proposition~\ref{transcendentallatticetriple-doublek3surface} and $\Omega_{\mathbb{Z}_2^3}^\perp\cong U(2)^{\oplus3}\oplus\langle-4\rangle^{\oplus2}$ by \cite[Proposition 5.1]{garbagnatisarti09}. Moreover, $(1,1,0,0,0,0)\in U\oplus U(2)\oplus\langle-4\rangle^{\oplus2}$ squares giving $2$ and $4\mid x^2$ for all $x\in U(2)^{\oplus3}\oplus\langle-4\rangle^{\oplus2}$. Therefore, $T_X$ cannot embed primitively into $\Omega_{\mathbb{Z}_2^3}^\perp$.

For part (ii), we first need some preliminaries. Let $M_{\mathbb{Z}_2^3}$ be the abstract lattice generated by the following vectors:
\begin{gather*}
m_1,\ldots,m_{14},\frac{\sum_{i=1}^8m_i}{2},\frac{\sum_{i=5}^{12}m_i}{2},\\
\frac{m_1+m_2+m_5+m_6+m_9+m_{10}+m_{13}+m_{14}}{2},
\end{gather*}
where $m_i^2=-2$ and $m_1\cdot m_j=0$ for all $i,j\in\{1,\ldots,14\}$, $i\neq j$. The lattice $M_{\mathbb{Z}_{2}^3}$ is negative definite of rank $14$ and has discriminant group $\mathbb{Z}_2^8$. By doing similar calculations to the ones we did to determine the discriminant quadratic form of $T_X$ given generators for $\NS(X)$, one can show that the discriminant quadratic form of $M_{\mathbb{Z}_{2}^3}^\perp$, the orthogonal complement of $M_{\mathbb{Z}_2^3}$ in $H^2(X;\mathbb{Z})$, is isomorphic to the discriminant quadratic form of $\langle2\rangle^{\oplus2}\oplus U(2)\oplus\langle-2\rangle^{\oplus4}$. This implies by Theorem~\ref{uniqueness2-elementarylattices} that $M_{\mathbb{Z}_2^3}^\perp$ is isometric to the lattice $\langle2\rangle^{\oplus2}\oplus U(2)\oplus\langle-2\rangle^{\oplus4}$. Now, $X$ is isomorphic to the minimal resolution of the quotient of a K3 surface by a symplectic action of the group $\mathbb{Z}_2^3$ if and only if the lattice $M_{\mathbb{Z}_2^3}$ embeds primitively into $\NS(X)$ (see \cite[Definition 2.5]{garbagnati16} and \cite[Corollary 8.9]{garbagnatisarti16}), or equivalently if and only if $T_X$ embeds primitively into $M_{\mathbb{Z}_2^3}^\perp$. But $T_X\cong U\oplus U(2)\oplus\langle-4\rangle^{\oplus2}$, and in particular there are two vectors $x,y\in T_X$ such that $x\cdot y=1$. However, for any two vectors $x,y\in M_{\mathbb{Z}_2^3}^\perp\cong\langle2\rangle^{\oplus2}\oplus U(2)\oplus\langle-2\rangle^{\oplus4}$, the product $x\cdot  y$ is even.
\end{proof}
\begin{corollary}
Let $X$ be a triple-double K3 surface with $\rho(X)=16$. Then $X$ does not contain $14$ disjoint smooth rational curves.
\end{corollary}
\begin{proof}
If $X$ contains $14$ disjoint smooth rational curves, then the lattice $M_{\mathbb{Z}_2^3}$ (see the proof of Proposition~\ref{XnoZ24andnoZ23symplecticactions}(ii)) would be primitively embedded in $\NS(X)$. But this would contradict Proposition~\ref{XnoZ24andnoZ23symplecticactions}(ii) (see \cite[Corollary 8.9]{garbagnatisarti16}).
\end{proof}


\section{$\mathbb{Z}_2^n$-covers of $\mathbb{P}^2$ branched along six lines and K3 surfaces}
\label{three4dimensionalfamiliesofk3surfacesrelated}


\subsection{Hirzebruch-Kummer coverings}
\begin{definition}
Let $n$ be a positive integer and let $D\subset\mathbb{P}^2$ be a line arrangement such that no point in $\mathbb{P}^2$ belongs to all the lines. Let $Y'\rightarrow\mathbb{P}^2$ be the normal finite covering such that its restriction to $\mathbb{P}^2\setminus D$ is the Galois unramified covering associated to the monodromy homomorphism
\begin{equation*}
H_1(\mathbb{P}^2\setminus D;\mathbb{Z})\rightarrow H_1(\mathbb{P}^2\setminus D;\mathbb{Z})\otimes\mathbb{Z}_n.
\end{equation*}
Then the minimal resolution $Y$ of $Y'$ is called the \emph{Hirzebruch-Kummer covering of exponent $n$} of $\mathbb{P}^2$ branched along the configuration of lines. An explicit construction of the Hirzebruch-Kummer covering can be found in \cite[Definition 63]{catanese}. It turns out $Y'$ is a $\mathbb{Z}_n^{r-1}$-cover of $\mathbb{P}^2$ branched along $D$, where $r$ equals the number of lines in $D$.
\end{definition}
\begin{example}
\label{hirzebruchkummercoveringexponent2p2sixlines}
We are interested in the Hirzebruch-Kummer covering of exponent $2$ of $\mathbb{P}^2$ branched along six general lines (which is a $\mathbb{Z}_2^5$-cover of $\mathbb{P}^2$ branched along the six lines). This is constructed as follows. Let $[W_0:\ldots:W_5]$ be coordinates in $\mathbb{P}^5$. Then, for general coefficients $a_i,b_i,c_i$, $i=0,\ldots,5$, let $Y\subset\mathbb{P}^5$ be the following intersection:
\begin{displaymath}
\left\{ \begin{array}{ll}
a_0W_0^2+\ldots+a_5W_5^2=0,\\
b_0W_0^2+\ldots+b_5W_5^2=0,\\
c_0W_0^2+\ldots+c_5W_5^2=0.
\end{array} \right.
\end{displaymath}
$Y$ is smooth for a general choice of the coefficients, hence $Y$ is a K3 surface. Consider the restriction to $Y$ of the morphism $\mathbb{P}^5\rightarrow\mathbb{P}^5$ given by
\begin{equation*}
[W_0:\ldots:W_5]\mapsto[W_0^2:\ldots:W_5^2].
\end{equation*}
Then this realizes $Y$ as a $\mathbb{Z}_2^5$-cover of the linear subspace $H\subset\mathbb{P}^5$ given by
\begin{displaymath}
\left\{ \begin{array}{ll}
a_0W_0+\ldots+a_5W_5=0,\\
b_0W_0+\ldots+b_5W_5=0,\\
c_0W_0+\ldots+c_5W_5=0,
\end{array} \right.
\end{displaymath}
which is isomorphic to $\mathbb{P}^2$. The branch locus of this cover is given by the restrictions to $H$ of the six coordinate hyperplanes of $\mathbb{P}^5$, which correspond to six general lines in $\mathbb{P}^2$. Label the six lines $\ell_0,\ldots,\ell_5$ and denote by $\pi$ the covering map $Y\rightarrow\mathbb{P}^2$. Then $\pi$ is the Hirzebruch-Kummer covering of exponent $2$ of $\mathbb{P}^2$ branched along $\ell_0,\ldots,\ell_5$. As the coefficients $a_i,b_j,c_k$ vary, $Y$ describes a $4$-dimensional family of K3 surfaces, which was also considered in \cite[Section 10.2]{garbagnatisarti16}.
\end{example}
\begin{remark}
\label{maximalcoverproperty}
There is no $\mathbb{Z}_2^n$-cover of $\mathbb{P}^2$ branched along $\ell_0,\ldots,\ell_5$ for $n\geq6$, and any $\mathbb{Z}_2^n$ cover of $\mathbb{P}^2$ branched along $\ell_0,\ldots,\ell_5$ for $n\leq4$ is an appropriate quotient of the Hirzebruch-Kummer covering $Y$ constructed above (see \cite[Section 4]{catanese}).
\end{remark}


\subsection{Relation with triple-double K3 surfaces}
\begin{observation}
\label{symplecticautomorphismshirzebruchkummercover}
Let $Y$ be the Hirzebruch-Kummer covering described in Example~\ref{hirzebruchkummercoveringexponent2p2sixlines}. Let $\iota_{01}$ be the restriction to $Y$ of the involution of $\mathbb{P}^5$ defined by
{\small
\begin{equation}
\label{symplecticinvolutiononhirzebruchkummercovering}
[W_0:W_1:W_2:W_3:W_4:W_5]\mapsto[-W_0:-W_1:W_2:W_3:W_4:W_5].
\end{equation}
}%
Define $\iota_{ij}$ analogously for all $i,j\in\{0,\ldots,5\}$, $i\neq j$. The involutions $\iota_{ij}$ form a subgroup of $\Aut(Y)$ isomorphic to $\mathbb{Z}_2^4$ acting symplectically on $Y$ because each involution has exactly eight fixed points. The eight fixed points of $\iota_{ij}$ are mapped to $\ell_i\cap\ell_j$ under $\pi$. One has that $\widetilde{Y/\mathbb{Z}_2^4}$ is the K3 surface given by the minimal resolution of the double cover of $\mathbb{P}^2$ branched along $\ell_0,\ldots,\ell_5$. This observation can also be found in \cite[Section 10.2]{garbagnatisarti16} and \cite[Section 4.3]{catanese}.
\end{observation}
\begin{proposition}
\label{triple-doublek3surfacesandthehirzebruchkummercover}
Let $X$ be a triple-double K3 surface, and let $X_3\rightarrow\mathbb{P}^2$ be the corresponding $\mathbb{Z}_2^3$-cover branched along the three pairs of lines $(\ell_0,\ell_1),(\ell_2,\ell_3),(\ell_4,\ell_5)$ (so that $X=\widetilde{X}_3$). Let $Y$ be the corresponding Hirzebruch-Kummer covering in Example~\ref{hirzebruchkummercoveringexponent2p2sixlines}. Consider the symplectic action on $Y$ of the group $\{\id_Y,\iota_{01},\iota_{23},\iota_{45}\}\cong\mathbb{Z}_2^2$ (see Observation~\ref{symplecticautomorphismshirzebruchkummercover}). Then $X\cong\widetilde{Y/\mathbb{Z}_2^2}$.
\end{proposition}
\begin{proof}
By Remark~\ref{maximalcoverproperty} we have that $X$ is isomorphic to the minimal resolution of the quotient of $Y$ by a subgroup $G\subset\mathbb{Z}_2^4$ of symplectic automorphisms. It follows that $G$ is isomorphic to $\mathbb{Z}_2^2$, and the only possibility for $G$ is to equal $\{\id_Y,\iota_{01},\iota_{23},\iota_{45}\}$. The reason for the latter statement is because this is the only choice so that $Y/G$ has exactly four singularities of type $A_1$ over each point $\ell_0\cap\ell_1,\ell_2\cap\ell_3,\ell_4\cap\ell_5$.
\end{proof}


\subsection{Other K3 surfaces which are $\mathbb{Z}_2^n$-covers of $\mathbb{P}^2$ branched along six lines}
\label{otherK3surfacesascovers}
Consider six general lines in $\mathbb{P}^2$ and denote by $D$ the divisor on $\mathbb{P}^2$ given by their sum. Let $Y$ be the Hirzebruch-Kummer covering of exponent $2$ of $\mathbb{P}^2$ branched along $D$ in Example~\ref{hirzebruchkummercoveringexponent2p2sixlines}. Let $X$ be a triple-double K3 surface obtained as a $\mathbb{Z}_2^3$-cover of $\mathbb{P}^2$ branched along $D$. Lastly, let $Z$ be the the minimal resolution of the double cover of $\mathbb{P}^2$ branched along $D$. As we remarked in the Introduction, the lattices $\NS(Y)$ and $\NS(Z)$ are well known for $\rho(Y)=\rho(Z)=16$, and the lattice $\NS(X)$ is computed in the current paper for $\rho(X)=16$. Let $X'$ be the minimal resolution of the quotient of $X$ by one of the three symplectic involutions $\iota_{110},\iota_{101},\iota_{011}$ in Proposition~\ref{involutiontriple-doublek3surface}. In Section~\ref{NSX'andTX'computed}, we compute $\NS(X')$ for $X'$ with minimal Picard number. The surface $X'$ can be viewed as an appropriate $\mathbb{Z}_2^2$ cover of $\mathbb{P}^2$ branched along $D$. There are other K3 surfaces different from $X,X',Y,Z$ which are $\mathbb{Z}_2^n$-covers of $\mathbb{P}^2$ branched along six general lines with $n=3,4$. We will investigate their N\'eron-Severi and transcendental lattices in future work.
\begin{remark}
If $X'$ is as above, after computing $\NS(X')$ and $T_{X'}$ in Section~\ref{NSX'andTX'computed} we realize that $X'$ admits $\mathbb{Z}_2^4$ symplectic action (see Proposition~\ref{X'hasZ24symplecticaction}). Since $\rho(X')=16$, this implies that $\NS(X')$ can be computed using \cite{garbagnatisarti16}. However, in Section~\ref{NSX'andTX'computed} we are able to find an explicit $\mathbb{Z}$-basis for $\NS(X')$ and relate it to the geometry of $X'$.
\end{remark}


\section{The N\'eron-Severi lattice of $X'$ with minimal Picard number}
\label{NSX'andTX'computed}


\subsection{Construction of the family}
\label{familyofquotients}
Let $\mathfrak{X}\rightarrow\mathcal{U}$ be the family of triple-double K3 surfaces in Definition~\ref{familyoftriple-doublek3surfaces}. The involution $\iota_{011}$ in Proposition~\ref{involutiontriple-doublek3surface} acts on $\mathfrak{X}$ inducing a symplectic action on each fiber of $\mathfrak{X}\rightarrow\mathcal{U}$ ($\iota_{110},\iota_{101}$ are treated similarly). Let $\mathfrak{X}'$ be the minimal resolution of the quotient of $\mathfrak{X}$ by $\iota_{011}$. Then $\mathfrak{X}'\rightarrow\mathcal{U}$ is a 4-dimensional family of K3 surfaces, and we denote by $X'$ one of its fibers. We compute the lattices $\NS(X')$ and $T_{X'}$ for $X'$ with minimal Picard number in several steps.


\subsection{A configuration of $20$ smooth rational curves on $X'$}
First of all, we want to understand what is the quotient by $\iota_{011}$ of the configuration of $24$ smooth rational curves $R_1,\ldots,R_{24}$ on $X$ in Figure~\ref{dualgraph24smoothrationalcurves} (observe that the eight fixed points of $\iota_{011}$ are in the complement of $R_1\cup\ldots\cup R_{24}$). Therefore, we need to understand how the involutions $\iota_{001}$ and $\iota_{010}$ act on these curves. In Figure~\ref{firsttwodoublecoverscurve}, we can see how $\iota_{001}$ acts by base change on the configuration fixing the eight curves with two branches on each one. Therefore, we have that
{\small
\begin{gather*}
\iota_{001}\cdot(1,\ldots,24)\\
=(3,4,1,2,7,8,5,6,9,10,11,12,15,16,13,14,19,20,17,18,21,22,23,24),
\end{gather*}
}%
where we only carry the indices of the curves $R_1,\ldots,R_{24}$ for simplicity. Similarly, we have that
{\small
\begin{gather*}
\iota_{010}\cdot(1,\ldots,24)\\
=(2,1,4,3,5,6,7,8,11,12,9,10,14,13,16,15,17,18,19,20,23,24,21,22),
\end{gather*}
}%
and hence
{\small
\begin{gather*}
\iota_{011}\cdot(1,\ldots,24)\\
=(4,3,2,1,7,8,5,6,11,12,9,10,16,15,14,13,19,20,17,18,23,24,21,22).
\end{gather*}
}%
In particular, we have the following classes of curves modulo $\iota_{011}$:
\begin{gather*}
\{R_1,R_4\},\{R_2,R_3\},\{R_5,R_7\},\{R_6,R_8\},\{R_9,R_{11}\},\{R_{10},R_{12}\},\\
\{R_{13},R_{16}\},\{R_{14},R_{15}\},\{R_{17},R_{19}\},\{R_{18},R_{20}\},\{R_{21},R_{23}\},\{R_{22},R_{24}\}.
\end{gather*}
Consider the commutative diagram (which we discussed in the proof of Proposition~\ref{involutionwithexactlyeightisolatedfixedpointsissymplectic})
\begin{center}
\begin{tikzpicture}[>=angle 90]
\matrix(a)[matrix of math nodes,
row sep=2em, column sep=2em,
text height=1.5ex, text depth=0.25ex]
{\Bl_8X&X'\\
X&X/\iota_{011},\\};
\path[->] (a-1-1) edge node[above]{$\pi$}(a-1-2);
\path[->] (a-1-1) edge node[left]{}(a-2-1);
\path[->] (a-1-2) edge node[left]{}(a-2-2);
\path[->] (a-2-1) edge node[left]{}(a-2-2);
\end{tikzpicture}
\end{center}
where $\Bl_8X$ is the blow up of $X$ at the eight fixed point of $\iota_{011}$. So $\pi\colon\Bl_8X\rightarrow X'$ is the $\mathbb{Z}_2$-cover of $X'$ branched along the eight exceptional divisors of the resolution $X'\rightarrow X/\iota_{011}$, which we denote by $N_1,\ldots,N_8$. If we set $C_1=\pi(R_1)=\pi(R_4)$, $C_2=\pi(R_2)=\pi(R_3)$, and so on, then the corresponding curve arrangement on $X'$ is shown in Figure~\ref{curvearrangementafterquotient}. In conclusion, we have two disjoint sets of smooth rational curves on $X'$ given by $\{C_1,\ldots,C_{12}\},\{N_1,\ldots,N_8\}$.
\begin{figure}
\centering
\includegraphics[scale=0.50,valign=t]{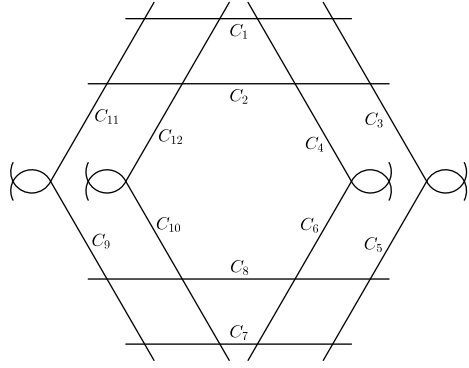}
\caption{Configuration of smooth rational curves on $X'$}
\label{curvearrangementafterquotient}
\end{figure}


\subsection{Even eights on $X'$}
With the notation introduced above, we know that
\begin{equation*}
N=\frac{N_1+\ldots+N_8}{2}\in\NS(X'),
\end{equation*}
and the lattice generated by $N_1,\ldots,N_8,\frac{N_1+\ldots+N_8}{2}$ is called the \emph{Nikulin lattice}. Let us find more sets of curves on $X'$ whose sum is $2$-divisible.

We recall the following result on Nikulin, which can be found in \cite[Section 2.3]{garbagnati16}. We state it in a form that is convenient for us.
\begin{proposition}
\label{minimalprimitivesublatticesdisjointsmoothrationalcurves}
Let $L_1,\ldots,L_k$ be smooth disjoint rational curves on a K3 surface $W$.
\begin{itemize}
\item If $k=13$, then, up to reordering the indices,
\begin{equation*}
\frac{\sum_{i=1}^8L_i}{2},~\frac{\sum_{i=5}^{12}L_i}{2}\in\NS(W).
\end{equation*}
\item If $k=14$, then, up to reordering the indices, the following vectors are in $\NS(W)$:
\begin{equation*}
\frac{\sum_{i=1}^8L_i}{2},~\frac{\sum_{i=5}^{12}L_i}{2},~\frac{L_1+L_2+L_5+L_6+L_9+L_{10}+L_{13}+L_{14}}{2}.
\end{equation*}
\end{itemize}
\end{proposition}
Let us apply Proposition~\ref{minimalprimitivesublatticesdisjointsmoothrationalcurves} for $k=13$ to the curves
\begin{equation*}
C_1,~C_5,~C_6,~C_9,~C_{10},~N_1,\ldots,~N_8.
\end{equation*}
We already know that $N\in\NS(X')$. This implies that, up to reordering the indices of the curves $N_i$, the sum of $N_1+N_2+N_3+N_4$ with other four curves among $C_1,C_5,C_6,C_9,C_{10}$ is $2$-divisible. If we use $C_1$ and $\{C_i,C_j,C_k\}\subset\{C_5,C_6,C_9,C_{10}\}$, then the intersection number
\begin{equation*}
\left(\frac{C_1+C_i+C_j+C_k+N_1+N_2+N_3+N_4}{2}\right)\cdot C_3,
\end{equation*}
is not an integer. Therefore, we have that
\begin{equation*}
\Lambda_1=\frac{C_5+C_6+C_9+C_{10}+N_1+N_2+N_3+N_4}{2}\in\NS(X').
\end{equation*}
But now let us apply Proposition~\ref{minimalprimitivesublatticesdisjointsmoothrationalcurves} for $k=14$ to the curves
\begin{equation*}
C_1,~C_2,~C_5,~C_6,~C_9,~C_{10},~N_1,\ldots,~N_8.
\end{equation*}
Then we have that
\begin{equation*}
\frac{C_1+C_2+C_i+C_j+N_1+N_2+N_5+N_6}{2}\in\NS(X'),
\end{equation*}
where $C_i,C_j$ is a choice of two curves among $C_5,C_6,C_9,C_{10}$. This is true up to permuting $N_1,\ldots,N_4$, and up to  permuting $N_5,\ldots,N_8$. Up to permuting $C_5,C_6,C_9,C_{10}$ we can assume that $\{i,j\}=\{5,6\}$, hence
\begin{equation*}
\Lambda_2=\frac{C_1+C_2+C_5+C_6+N_1+N_2+N_5+N_6}{2}\in\NS(X').
\end{equation*}


\subsection{Computation of $\NS(X')$ and $T_{X'}$}
\begin{theorem}
\label{allaboutnsx'andtx'}
Let $X'$ be a fiber of the family of K3 surfaces $\mathfrak{X}'\rightarrow\mathcal{U}$ defined in Section~\ref{familyofquotients}. Assume $X'$ has minimal Picard number. Then the lattice $\NS(X')$ has rank $16$ and discriminant group $\mathbb{Z}_2^4\oplus\mathbb{Z}_4^2$. A $\mathbb{Z}$-basis for $\NS(X')$ is given by
\begin{equation*}
\{C_1,C_2,C_3,C_4,C_5,C_7,C_8,C_9,N_1,N_2,N_3,N_5,N_7,N,\Lambda_1,\Lambda_2\}.
\end{equation*}
The discriminant quadratic form of $\NS(X')$ is given by
\begin{displaymath}
\left( \begin{array}{cccccc}
0&\frac{1}{2}&0&0&0&0\\
\frac{1}{2}&0&0&0&0&0\\
0&0&0&\frac{1}{2}&0&0\\
0&0&\frac{1}{2}&0&0&0\\
0&0&0&0&\frac{1}{4}&0\\
0&0&0&0&0&\frac{1}{4}
\end{array} \right).
\end{displaymath}
Hence, the transcendental lattice $T_{X'}$ is isometric to $U(2)^{\oplus2}\oplus\langle-4\rangle^{\oplus2}$.
\end{theorem}
\begin{proof}
Some of the computations that follow are computer assisted (we used \cite{wolfram}). Define
\begin{equation*}
M=\langle C_1,\ldots,C_{12},N_1,\ldots,N_8,N,\Lambda_1,\Lambda_2\rangle_\mathbb{Z}\subseteq\NS(X').
\end{equation*}
The lattice $M$ has rank $16$, implying that $\rho(X')\geq16$. But since the K3 surfaces $X'$ vary in a $4$-dimensional family, we have that the minimum Picard number for $X'$ is $16$. Therefore, $M$ is a sublattice of $\NS(X')$ of finite index. We show that $M=\NS(X')$.

A $\mathbb{Z}$-basis for $M$ is given by
\begin{equation*}
\{C_1,C_2,C_3,C_4,C_5,C_7,C_8,C_9,N_1,N_2,N_3,N_5,N_7,N,\Lambda_1,\Lambda_2\}.
\end{equation*}
This is true because
\begin{gather*}
C_6=C_3-C_4+C_5,~C_{10}=C_1+C_2+C_3+C_4-C_7-C_8-C_9,\\
C_{11}=C_3+C_5-C_9,~C_{12}=-C_1-C_2-C_4+C_5+C_7+C_8+C_9,\\
N_4=-C_1-C_2-2C_3-2C_5+C_7+C_8-N_1-N_2-N_3+2\Lambda_1,\\
N_6=-C_1-C_2-C_3+C_4-2C_5-N_1-N_2-N_5+2\Lambda_2,\\
N_8=2C_1+2C_2+3C_3-C_4+4C_5-C_7-C_8+N_1+N_2-N_7\\
+2N-2\Lambda_1-2\Lambda_2.
\end{gather*}

To show that $M=\NS(X')$ we use the same strategy we used in the proof of Theorem~\ref{mainresultNSgeneratedby24smoothrationalcurves}: we prove that any isotropic element of the discriminant group $A_M$ cannot be an element in $\NS(X')$. If $B$ is the intersection matrix of the curves in the $\mathbb{Z}$-basis of $M$ above, then the dual $M^*$ is generated over $\mathbb{Z}$ by the rows of $B^{-1}$. Using Sage, we can find matrices $M_1,M_2\in\SL_{16}(\mathbb{Z})$ such that $M_1B^{-1}M_2$ is the diagonal matrix $\diag\left(1,\ldots,1,\frac{1}{2},\frac{1}{2},\frac{1}{2},\frac{1}{2},\frac{1}{4},\frac{1}{4}\right)$. This tells us that the discriminant group of $M$ is isomorphic to $\mathbb{Z}_2^4\oplus\mathbb{Z}_4^2$. In particular, the rows of $M_1B^{-1}$ give us an alternative $\mathbb{Z}$-basis of $M^*$ to work with. The matrix $M_1B^{-1}$ is explicitly given by
{\scriptsize
\begin{displaymath}
\left( \begin{array}{cccccccccccccccc}
0&0&0&0&0&0&0&0&0&0&0&0&0&0&0&1\\
0&0&0&0&0&0&0&0&0&0&0&0&0&0&1&0\\
0&0&0&0&0&0&0&0&0&0&0&0&0&1&0&0\\
0&0&0&0&0&0&0&0&0&0&0&0&1&0&0&0\\
0&0&0&0&0&0&0&0&0&0&0&1&0&0&0&0\\
0&0&0&0&0&0&0&0&0&0&1&0&0&0&0&0\\
0&0&0&0&0&0&0&1&0&0&0&0&0&0&0&0\\
0&0&0&0&0&0&1&0&0&0&0&0&0&0&0&0\\
0&0&0&1&0&0&0&0&0&0&0&0&0&0&0&0\\
0&1&0&0&0&0&0&0&0&0&0&0&0&0&0&0\\
0&0&0&0&0&0&0&0&0&\frac{1}{2}&-\frac{1}{2}&-\frac{1}{2}&-\frac{1}{2}&0&0&0\\
0&0&0&0&0&0&0&0&\frac{1}{2}&0&-\frac{1}{2}&-\frac{1}{2}&-\frac{1}{2}&0&0&0\\
0&0&0&0&\frac{1}{2}&0&0&-\frac{1}{2}&0&0&0&-\frac{1}{2}&-\frac{1}{2}&0&0&0\\
0&0&\frac{1}{2}&-\frac{1}{2}&0&0&0&0&0&0&0&0&0&0&0&0\\
0&0&0&0&0&\frac{1}{4}&-\frac{1}{4}&-\frac{1}{2}&0&0&-\frac{1}{2}&0&-\frac{1}{2}&0&0&0\\
\frac{1}{4}&-\frac{1}{4}&0&-\frac{1}{2}&0&0&0&0&0&0&-\frac{1}{2}&0&-\frac{1}{2}&0&0&0
\end{array} \right).
\end{displaymath}
}%
Denote by $v_1,v_2,v_3,v_4,w_1,w_2$ respectively the last six rows of the matrix above, which generate the discriminant group $A_M$. We can then enumerate the isotropic elements of $A_M$, which are given by the classes modulo $M$ of the following vectors:
{\footnotesize
\[
(C_{1}+C_{2}+C_{7}+C_{8})/2,~(C_{1}+C_{2}+C_{3}+C_{4})/2,~(C_{3}+C_{4}+C_{7}+C_{8})/2,
\]
\[
(C_{5}+C_{9}+N_{5}+N_{7})/2,~(N_{1}+N_{3}+N_{5}+N_{7})/2,~(C_{5}+C_{9}+N_{1}+N_{3})/2,
\]
\[
(N_{2}+N_{3}+N_{5}+N_{7})/2,~(C_{5}+C_{9}+N_{2}+N_{3})/2,~(C_{1}+C_{2}+N_{1}+N_{2})/2,
\]
\[
(C_{7}+C_{8}+N_{1}+N_{2})/2,~(C_{3}+C_{4}+N_{1}+N_{2})/2,
\]
\[
(C_{3}+C_{4}+C_{5}+C_{9}+N_{5}+N_{7})/2,
\]
\[
(C_{3}+C_{4}+C_{5}+C_{9}+N_{1}+N_{3})/2,
\]
\[
(C_{3}+C_{4}+C_{5}+C_{9}+N_{2}+N_{3})/2,
\]
\[
(C_{1}+C_{2}+C_{5}+C_{7}+C_{8}+C_{9}+N_{5}+N_{7})/2,
\]
\[
(C_{1}+C_{2}+C_{7}+C_{8}+N_{1}+N_{3}+N_{5}+N_{7})/2,
\]
\[
(C_{1}+C_{2}+C_{3}+C_{4}+N_{1}+N_{3}+N_{5}+N_{7})/2,
\]
\[
(C_{3}+C_{4}+C_{7}+C_{8}+N_{1}+N_{3}+N_{5}+N_{7})/2,
\]
\[
(C_{1}+C_{2}+C_{5}+C_{7}+C_{8}+C_{9}+N_{1}+N_{3})/2,
\]
\[
(C_{1}+C_{2}+C_{7}+C_{8}+N_{2}+N_{3}+N_{5}+N_{7})/2,
\]
\[
(C_{1}+C_{2}+C_{3}+C_{4}+N_{2}+N_{3}+N_{5}+N_{7})/2,
\]
\[
(C_{3}+C_{4}+C_{7}+C_{8}+N_{2}+N_{3}+N_{5}+N_{7})/2,
\]
\[
(C_{1}+C_{2}+C_{5}+C_{7}+C_{8}+C_{9}+N_{2}+N_{3})/2,
\]
\[
(C_{1}+C_{2}+C_{3}+C_{4}+C_{7}+C_{8}+N_{1}+N_{2})/2,
\]
\[
(C_{1}+C_{2}+C_{5}+C_{9}+N_{1}+N_{2}+N_{5}+N_{7})/2,
\]
\[
(C_{5}+C_{7}+C_{8}+C_{9}+N_{1}+N_{2}+N_{5}+N_{7})/2,
\]
\[
(C_{1}+C_{2}+C_{3}+C_{4}+C_{5}+C_{7}+C_{8}+C_{9}+N_{5}+N_{7})/2,
\]
\[
(C_{1}+C_{2}+C_{3}+C_{4}+C_{5}+C_{7}+C_{8}+C_{9}+N_{1}+N_{3})/2,
\]
\[
(C_{1}+C_{2}+C_{3}+C_{4}+C_{5}+C_{7}+C_{8}+C_{9}+N_{2}+N_{3})/2,
\]
\[
(C_{1}+C_{2}+C_{3}+C_{4}+C_{5}+C_{9}+N_{1}+N_{2}+N_{5}+N_{7})/2,
\]
\[
(C_{3}+C_{4}+C_{5}+C_{7}+C_{8}+C_{9}+N_{1}+N_{2}+N_{5}+N_{7})/2.
\]
}%
It is easy to show that all the elements above are not vectors of the lattice $\NS(X')$. To show this, let $v$ be one of these vectors, and assume by contradiction that $v\in\NS(X')$. Then we can
\begin{itemize}
\item Add or subtract elements of $M$ to $v$;
\item Use the relations
\begin{gather*}
C_1+C_2+C_3+C_4=C_7+C_8+C_9+C_{10},\\
C_1+C_2+C_{11}+C_{12}=C_5+C_6+C_7+C_8,\\
C_3+C_5=C_4+C_6=C_9+C_{11}=C_{10}+C_{12}.
\end{gather*}
\end{itemize}
Using these operations on $v$ we can produce a vector $v'\in\NS(X')$ equal to half the sum of four disjoint smooth rational curves (see Example~\ref{illustrationapplicationoperations} after the end of this proof), which is impossible by Theorem~\ref{theoremaboutcardinalitiesofevensets} (observe that, in some cases, $v$ is already half the sum of four disjoint smooth rational curves). This implies that $M=\NS(X')$.

If we choose $\{v_1,v_2,v_4+2w_2,v_3+v_4,w_1,w_2\}$ as basis for the discriminant group of $\NS(X')$, we obtain the discriminant quadratic form claimed in the statement of Theorem~\ref{allaboutnsx'andtx'}. By Theorem~\ref{sublatticeandperpsamediscriminantgroup} and \cite[Chapter VIII, Corollary 7.8(3)]{mirandamorrison} it follows that $T_{X'}$ is isometric to the lattice $U(2)^{\oplus2}\oplus\langle-4\rangle^{\oplus2}$.
\end{proof}
\begin{example}
\label{illustrationapplicationoperations}
Say $v=(C_{1}+C_{2}+C_{7}+C_{8}+N_{2}+N_{3}+N_{5}+N_{7})/2$. Every time we use one of the operations above, we write ``$\approx$". Then
\begin{gather*}
\frac{C_{1}+C_{2}+C_{7}+C_{8}+N_{2}+N_{3}+N_{5}+N_{7}}{2}\\
\approx\frac{C_{3}+C_{4}+C_{9}+C_{10}+N_{2}+N_{3}+N_{5}+N_{7}}{2}\\
\approx\frac{C_{5}+C_{6}+C_{9}+C_{10}+N_{2}+N_{3}+N_{5}+N_{7}}{2}\approx\frac{N_{1}+N_{4}+N_{5}+N_{7}}{2}.
\end{gather*}
\end{example}


\subsection{Additional properties of $X'$}
The following proposition is the analogue of Proposition~\ref{XnoZ24andnoZ23symplecticactions} for the K3 surfaces $X'$.
\begin{proposition}
\label{X'hasZ24symplecticaction}
Let $X'$ be a fiber of the family of K3 surfaces $\mathfrak{X}'\rightarrow\mathcal{U}$ defined in Section~\ref{familyofquotients} with $\rho(X')=16$. Then the following hold:
\begin{itemize}
\item[(i)] $X'$ admits $\mathbb{Z}_2^4$ symplectic action;
\item[(ii)] $X'$ is not isomorphic to the minimal resolution of the quotient of a K3 surface by a symplectic action of the group $\mathbb{Z}_2^4$;
\item[(iii)] $X'$ is isomorphic to the minimal resolution of the quotient of a K3 surface by a symplectic action of the group $\mathbb{Z}_2^3$.
\end{itemize}
\end{proposition}
\begin{proof}
We know from \cite{nikulin79} that $X'$ admits $\mathbb{Z}_2^4$ symplectic action if and only if the lattice $T_{X'}$ embeds primitively into $\Omega_{\mathbb{Z}_2^4}^\perp$. But $T_{X'}\cong U(2)^{\oplus2}\oplus\langle-4\rangle^{\oplus2}$ by Theorem~\ref{allaboutnsx'andtx'}, and $\Omega_{\mathbb{Z}_2^4}^\perp\cong U(2)^{\oplus3}\oplus\langle-8\rangle$ by \cite[Proposition 5.1]{garbagnatisarti09}. Therefore, it is enough to show that $\langle-4\rangle^{\oplus2}$ embeds primitively into $U(2)\oplus\langle-8\rangle$. The vectors $(1,1,1),(-1,1,0)$ in $U(2)\oplus\langle-8\rangle$ generate a sublattice $L$ isometric to $\langle-4\rangle^{\oplus2}$. To show that $L$ is primitive, assume that $m(x,y,z)=(a-b,a+b,a)\in L$ for some integers $a,b,m$ with $m>1$ and $(x,y,z)\in U(2)\oplus\langle-8\rangle$. Then $x=2z-y$, which implies $(x,y,z)=(z)(1,1,1)+(y-z)(-1,1,0)\in L$, proving part (i).

Part (ii) follows from the fact that $T_{X'}$ is not isometric to any of the transcendental lattices listed in \cite[Theorem 8.3]{garbagnatisarti16}.

For part (iii), we know from the discussion in Section~\ref{three4dimensionalfamiliesofk3surfacesrelated} that $X'$ is the minimal resolution of the quotient of the Hirzebruch-Kummer covering $Y$ (see Example~\ref{hirzebruchkummercoveringexponent2p2sixlines}) by a symplectic action of $\mathbb{Z}_2^3$.
\end{proof}



\

{\small\textsc{Department of Mathematics and Statistics, University of Massachusetts Amherst, Amherst, MA 01003, USA}}

\emph{E-mail address:} \url{schaffler@math.umass.edu}

\end{document}